\documentclass[a4paper,11pt,leqno]{article}

\usepackage[latin1]{inputenc}
\usepackage[T1]{fontenc} 
\usepackage[english]{babel}
\usepackage{verbatim}
\usepackage{amsmath}
\usepackage{amsthm}
\usepackage{amssymb}
\usepackage{bm}

\theoremstyle{definition}
\newtheorem{defin}{Definition}[section]
\newtheorem{ex}[defin]{Example}
\newtheorem{rem}[defin]{Remark}

\theoremstyle{plane}
\newtheorem{thm}[defin]{Theorem}
\newtheorem{prop}[defin]{Proposition}
\newtheorem{coroll}[defin]{Corollary}
\newtheorem{lemma}[defin]{Lemma}

\newcommand{\mbb}{\mathbb}

\newcommand{\mc}{\mathcal}
\newcommand{\veps}{\varepsilon}

\newcommand{\wtilde}{\widetilde}
\newcommand{\vphi}{\varphi}
\newcommand{\oline}{\overline}

\newcommand{\ra}{\rightarrow}

\newcommand{\hra}{\hookrightarrow}

\newcommand{\g}{\gamma}

\newcommand{\R}{\mathbb{R}}

\newcommand{\N}{\mathbb{N}}
\newcommand{\Z}{\mathbb{Z}}

\newcommand{\Id}{{\rm Id}\,}

\def\d{\partial}

\title{\Large{\bfseries{\textsc{Weak observability estimates for $1$-D wave equations \\
with rough coefficients}}}}

\author{ \textsl{Francesco Fanelli}$\,^1$ $\quad$ and $\quad$ \textsl{Enrique Zuazua}$\,^{1,2}$  \vspace{.3cm} \\
{\small $\,^1\,$ \textsc{BCAM - Basque Center for Applied Mathematics}} \\
{\footnotesize Alameda de Mazarredo, 14 -- E48009 Bilbao - Basque Country, SPAIN}  \vspace{.2cm} \\
{\small $\,^2\,$ \textsc{Ikerbasque - Basque Foundation for Science}} \\
{\footnotesize Alameda Urquijo, 36-5, Plaza Bizkaia -- E48011 Bilbao - Basque Country, SPAIN}  \vspace{.2cm} \\
{\footnotesize\ttfamily{ffanelli@bcamath.org $\;$ -- $\;$ zuazua@bcamath.org}} \vspace{0.3cm} }

\date\today

\begin{document}

\maketitle

\subsubsection*{Abstract}
{\small In this paper we prove observability estimates for  $1$-dimensional wave equations with non-Lipschitz coefficients.
For coefficients in  the Zygmund class we prove a ``classical'' observability estimate, which extends the well-known observability results in the energy space for $BV$ regularity.
When the coefficients are instead log-Lipschitz or log-Zygmund, we prove observability estimates ``with loss of derivatives'': in order to estimate the total energy of the solutions, we need measurements on some higher order Sobolev norms at the boundary. This last result represents the intermediate step between
the Lipschitz (or Zygmund) case, when observability estimates hold in the energy space, and the H\"older one, when they fail at any finite order (as proved in \cite{Castro-Z}) due to an infinite loss of derivatives.
We  also establish a sharp relation between the modulus of continuity of the coefficients and the loss of derivatives in the observability estimates. In particular, we will show that under any condition which is weaker than the log-Lipschitz one (not only H\"older, for instance), observability estimates fail in general, while in the intermediate instance between the Lipschitz and the log-Lipschitz ones they can hold only admitting a loss of a finite number of derivatives. This classification has an exact counterpart when considering also the second variation of the coefficients. }


\section{Introduction and main results} \label{s:intro-results}

\subsection{Motivations}

The present paper is devoted to prove boundary observability estimates for the $1$-D wave equation with non-Lipschitz coefficients. 

On the one hand, for the first time we will
look not only at the first, but also at the second variation of the coefficients. Besides, we will consider Zygmund type regularity conditions.

On the other hand, the results we will get are twofold: we will prove not only ``classical'' observability estimates in the natural energy spaces, but also (under precise hypotheses) observability estimates
``with loss of derivatives'', which are, to our knowledge, new in this context.

\medbreak
Let us recall that boundary (or internal) observability estimates are equivalent to the controllability property of the system under the action of a control on the boundary (or in the interior respectively) of the domain. 

Our analysis is limited to coefficients just depending on the space variable $x$. The general case of coefficients depending both on $t$ and $x$ is much more delicate and it will be matter of future studies.

\medbreak
So, let $\Omega=]0,1[$ and $T>0$ (possibly $T=+\infty$) and consider the $1$-D wave equation
$$
\left\{\begin{array}{ll}
        \rho(x)\,\d_t^2u\,-\,\d_x\bigl(a(x)\,\d_xu\bigr)\,=\,0 & \quad\mbox{ in }\;\Omega\times\,]0,T[ \\[1ex]
        u(t,0)\,=\,u(t,1)\,=\,0 & \quad\mbox{ in }\;]0,T[ \\[1ex]
        u(0,x)\,=\,u_0(x)\,,\quad \d_tu(0,x)\,=\,u_1(x) & \quad\mbox{ in }\;\Omega\,,
       \end{array} 
\right. \leqno(W\!E)
$$
where the coefficients satisfy boundedness and strictly hyperbolicity conditions:
$$
0\,<\,\rho_*\,\leq\,\rho(x)\,\leq\,\rho^*\;,\qquad0\,<\,a_*\,\leq\,a(x)\,\leq\,a^*\,. \leqno(H)
$$

Under these hypotheses, system $(W\!E)$ is well-posed in $H^1_0 (\Omega) \times L^2(\Omega)$: for any initial data $(u_0,u_1)\in H^1_0(\Omega)\times L^2(\Omega)$, there exists
a unique $u\in\mc{C}([0,T];H^1_0(\Omega))\cap\mc{C}^1([0,T];L^2(\Omega))$ solution to $(W\!E)$. In addition, its energy
$$
E(t)\,:=\,\frac{1}{2}\,\int_\Omega\biggl(\rho(x)\,\left|u_t(t,x)\right|^2\,+\,a(x)\,\left|u_x(t,x)\right|^2\biggr)\,dx
$$
is preserved during the evolution in time: $E(t)\equiv E(0)$ in $[0,T]$.

Moreover, if the coefficients are smooth (say e.g. Lipschitz continuous) the following observability properties hold true with $T^*:=\left\|\sqrt{\rho/a}\right\|_{L^1}$:
\begin{itemize}
\item \emph{Internal observability}: for any $\Theta:=\,]\ell_1,\ell_2[\,\subset\Omega$ and any $T>2T^*\max\{\ell_1,1-\ell_2\}$, there exists a constant $C_i>0$ such that
\begin{equation} \label{intro-est:int_obs}
E(0)\,\leq\,C_i\,\int^T_0\int^{\ell_2}_{\ell_1}\biggl(\rho(x)\,\left|u_t(t,x)\right|^2\,+\,a(x)\,\left|u_x(t,x)\right|^2\biggr)\,dx\,dt\,;
\end{equation}
\item \emph{Boundary observability}: for any $T>T^*$, there exists a constant $C_b>0$ such that
\begin{equation} \label{intro-est:bound_obs}
E(0)\,\leq\,C_b\,\int^T_0\biggl(\left|u_x(t,0)\right|^2\,+\,\left|u_x(t,1)\right|^2\biggr)\,dt
\end{equation}
($T>2T^*$ if, in the right-hand side, one considers just the observation at $x=0$ or $x=1$).
\end{itemize}
Let us stress the fact that the conditions on the observability time $T$ are necessary: due to the finite propagation speed for $(W\!E)$, no observability estimate can hold true if $T$ is
too small (see e.g. \cite{JLLions}). The constants in the previous inequalities can depend on $T$ and on $\ell_1$, $\ell_2$, but are independent of the solution.

The previous estimates can be proved by means of a genuinely $1$-dimensional technique, i.e. the \emph{sidewise energy estimates}, where the role of time and space are interchanged
(see e.g. \cite{Cox-Z}).

In higher space dimensions the problem is much more complex and other techniques are required. In \cite{B-L-R} Bardos, Lebeau and Rauch (see also  \cite{Burq})
 proved that a necessary and sufficient condition for observability estimates to hold for wave equations with smooth coefficients  is that the observability region $\Theta\subset\oline{\Omega}$ satisfy the so-called ``Geometric Control Condition'' (GCC for brevity). Roughly speaking, this means that $\Theta$ has to absorbe all the rays of geometric optics in time $T$. The necessity of this property can be seen by constructing the so called gaussian beam solutions. 

When coefficients of the wave equation under consideration are constant the rays are straight lines but for variable  coefficients  rays may have a complex dynamics, and concentration phenomena may occur making the GCC  harder to be verified. In the $1$-D case, instead, the situation is easier since rays can only travel in a sidewise manner. 

The issue is widely open for wave equations with low regularity coefficients. In this paper we address the $1-d$ case rather exhaustively. As we shall show, the lack of regularity of coefficients may cause a loss of derivatives of solutions along propagation, thus leading to weaker observability estimates or even to its failure.

\medbreak
For system $(W\!E)$, observability estimates \eqref{intro-est:int_obs}, \eqref{intro-est:bound_obs} in the energy space are well-known to be  verified when the coefficients are in the $BV$ class (see \cite{FC-Z}). In this paper we analyze  this issue for less regular coefficients.
Note however that, in \cite{Castro-Z},  the $BV$ assumption on the coefficients was shown to be  somehow sharp, proving that both interior and boundary observability estimates may fail for H\"older continuous coefficients, due to an infinite loss of derivatives (see also Theorem \ref{th:LL-LZ} below for the case of finite loss). Resorting to the original ideas of the work \cite{C-S} by Colombini and Spagnolo (which actually go back to paper \cite{C-DG-S}), in \cite{Castro-Z} an explicit counterexample was constructed for which such a phenomenon could be observed. The fundamental issue was making the coefficients oscillate more and more approaching to the extreme values $x=0$ and $x=1$ to build solutions exponentially concentrated far from $\partial\Omega$. Therefore,
it is impossible to get observability estimates at the boundary, or in the interior whenever the subinterval $[\ell_1,\ell_2]$ is different from the whole $\Omega$.
A simpler result in that spirit was earlier proved in \cite{A-B-R} showing that \eqref{intro-est:int_obs} and \eqref{intro-est:bound_obs} can fail to be uniform in the context of homogenization, for families of uniformly bounded (above and below) but rapidly oscillating coefficients.

\medbreak
The present paper sets in such a context: it aims to complete, as precisely as possible, the general picture of boundary observability estimates for wave equations with non-regular coefficients
in the $1$-D case.

Before stating our results, let us make some harmless but useful changes, in order to simplify the presentation. First of all, up to extending the coefficients outside the domain under consideration (we will be more clear in section \ref{s:tools} about this point), we can suppose $\rho$ and $a$ to be defined on the whole $\R_x$.

Moreover, in the whole paper, we will deal with $\rho$, $a\,\in\,\mc{A}$, where $\mc{A}$ is a Banach algebra; moreover, thanks to $(H)$, also
$1/\rho$ and $1/a$ will belong to $\mc{A}$ (see also subsection \ref{ss:coeff}). So, up to performing the change of variables
$$
y\,=\,\phi(x)\,:=\,\int^x_0\frac{1}{a(\zeta)}\,d\zeta
$$
and defining the new unknown $\wtilde{u}(y):=u\circ\phi^{-1}(y)$, without loss of generality, it is
 enough to consider the following system:
\begin{equation} \label{eq:we}
\left\{\begin{array}{ll}
        \omega(x)\d_t^2u\,-\,\d^2_xu\,=\,0 & \quad\mbox{ in }\;\Omega\times\,]0,T[ \\[1ex]
        u(t,0)\,=\,u(t,1)\,=\,0 & \quad\mbox{ in }\;]0,T[ \\[1ex]
        u(0,x)\,=\,u_0(x)\,,\quad \d_tu(0,x)\,=\,u_1(x) & \quad\mbox{ in }\;\Omega\,,
       \end{array}
\right.
\end{equation}
where $\omega$ still belongs to $\mc{A}$ and satisfies, for some constants $\omega_*$ and $\omega^*$,
\begin{equation} \label{eq:hyp}
0\,<\,\omega_*\,\leq\,\omega(x)\,\leq\,\omega^*\,.
\end{equation}
Associated to the coefficient  $\omega$, let us also define the time
\begin{equation} \label{def:vel}
T_\omega\,:=\,\int_\Omega\sqrt{\omega(x)}\,dx\,.
\end{equation}
With the terminology we will introduce in section \ref{s:proof}, by Theorem
 4 of \cite{C-DG-S} $T_\omega$ is the maximal displacement of a sidewise wave in the domain $\Omega$.


Finally, in proving our results we will work with smooth data and solutions of equation \eqref{eq:we}.
Recovering the corresponding results for ``critical regularity'' initial data follows then by standard density argument, thanks to a priori bounds depending just on the relevant norms.

\subsection{Main results}

Now we are ready for stating our main results.


First of all, let us assume $\omega$ to fulfill an integral Zygmund assumption (see also Subection \ref{ss:coeff}) on $\Omega$: there exists a constant $K>0$ such that, for all $0<h<1/2$,
\begin{equation} 
\int_{h}^{1-h}\left|\omega(x+h)\,+\,\omega(x-h)\,-\,2\,\omega(x)\right|\,dx\;\leq\;K\,h\,. \label{est:Z_1}
\end{equation}
We define also $|\omega|_{Z}$ as the infimum of the constants $K$ for which \eqref{est:Z_1} holds true.

\begin{thm} \label{th:Z}
Let us consider the strictly hyperbolic problem \eqref{eq:we}-\eqref{eq:hyp}, with $T>2T_\omega$ and $\omega$ satisfying relation \eqref{est:Z_1}. Assume the inital data $(u_0,u_1)\in H^1_0(\Omega)\times L^2(\Omega)$.

Then, there exists a constant $C$ (depending only on $T$, $\omega_*$, $\omega^*$ and $|\omega|_Z$) such that
\begin{equation} \label{est:obs_Z}
\left\|u_0\right\|^2_{H^1_0(\Omega)}\,+\,\left\|u_1\right\|^2_{L^2(\Omega)}\,\leq\,C\,\int^T_0\left|\d_xu(t,0)\right|^2\,dt\,.
\end{equation}
\end{thm}

Let us make some useful comments on the previous statement.

\begin{rem} \label{r:Z}
\begin{itemize} 
\item[(i)] Under the hypothesis of Theorem \ref{th:Z}, for initial data $(u_0,u_1)$ in the energy space $H^1_0\times L^2$ there exists a unique global in time solution $u$ to \eqref{eq:we} such that
$$
u\,\in\,L^\infty([0,T];H^1_0(\Omega))\,\cap\,W^{1,\infty}([0,T];L^2(\Omega))
$$
(see for instance \cite{H-S}). In particular, $\d_t^2u\in W^{-1,\infty}([0,T];L^2(\Omega))$ and,
as $\omega\in L^\infty(\Omega)$,
$$
\d_x^2u\,=\,\omega\,\d_t^2u\;\in\;W^{-1,\infty}([0,T];L^2(\Omega))\,.
$$
Therefore, the trace of $\d_xu$ at any point $x$, and in particular at $x=0$ is well-defined. Moreover, by ``reversing'' inequality \eqref{est:obs_Z}, in subsection \ref{ss:E} we will show that its right-hand side is finite, and all its terms have sense. 

\item[(ii)] Let us note that condition \eqref{est:Z_1} is weaker than the $W^{1,1}$ one (see also Corollary \ref{c:Z-BV} and Example \ref{ex:Z} below). Therefore,  Theorem \ref{th:Z} improves the previous result for $BV$ coefficients (see paper \cite{FC-Z}).

The fact is that the $BV$ hypothesis on the coefficients  is sharp if one just considers their first variation. A condition like  \eqref{est:Z} is related (for regular functions) with their second derivative. Thus our assumptions allow to consider larger classes of coefficients.

In Section \ref{s:sharp} we will give motivations in order
to support the strenght of the previous result.
\end{itemize}
\end{rem}

Let us now deal with lower regularity coefficients, and consider integral log-Lipschitz and log-Zygmund conditions. These assumptions read, respectively, as follows: there exists a constant $K>0$ such that, for all $0<h<1/2$,
\begin{eqnarray}
\int_0^{1-h}\left|\omega(x+h)\,-\,\omega(x)\right|\,dx & \leq & K\,h\,\log\left(1+\frac{1}{h}\right) \label{est:LL_1} \\
\int_h^{1-h}\left|\omega(x+h)\,+\,\omega(x-h)\,-\,2\,\omega(x)\right|\,dx & \leq & K\,h\,\log\left(1+\frac{1}{h}\right)\,. \label{est:LZ_1}
\end{eqnarray}
We also set $|\omega|_{LL}$ and $|\omega|_{LZ}$ as the infimum of the constants $K$ for which, respectively, inequalities \eqref{est:LL_1} and \eqref{est:LZ_1} hold true.

Under any of these hypotheses, it is
 possible to show an observability estimate, provided that we allow a loss of a finite number of derivatives. Roughly speaking, in order to control the energy of the initial data, we need to add also the contributions coming from the $L^2$-norms of the time derivatives of the solution at the boundary $x=0$.

For notation convenience, let us introduce the operator
\begin{equation} \label{def:deriv}
D_\omega\,:\quad f\;\longmapsto\;\frac{1}{\omega(x)}\,\d_x^2f\,.
\end{equation}
As usual, we will denote with $D_\omega^k$ the composition $D_\omega\circ\ldots\circ D_\omega$ ($k$ times) for any $k>0$, with the convention $D_\omega^0=\Id$.

\begin{thm} \label{th:LL-LZ}
Let us consider the strictly hyperbolic problem \eqref{eq:we}-\eqref{eq:hyp}, with $T>2T_\omega$. Assume this time $\omega$ to satisfy relation \eqref{est:LL_1} or \eqref{est:LZ_1}.

Then, there exist two positive constants $C$ and $m\in\N$, depending only on  $\omega_*$, $\omega^*$ and $|\omega|_{LL}$ or $|\omega|_{LZ}$ respectively ($C$ also depends on $T$) for which
\begin{equation} \label{est:obs_log}
\left\|u_0\right\|^2_{H^1_0(\Omega)}\,+\,\left\|u_1\right\|^2_{L^2(\Omega)}\,\leq\,C\,\int^T_0\bigl|\d_t^m\d_xu(t,0)\bigr|^2\,dt\,,
\end{equation}
holds true for any initial data $u_0\in H^{2m+1}(\Omega)\cap H^1_0(\Omega)$ and $u_1\in H^{2m}(\Omega)$ such that
\begin{equation} \label{eq:D^m_omega}
D_\omega^m u_0\,\in\,H^1(\Omega)\qquad\mbox{ and }\qquad
D^m_\omega u_1\,\in\,L^2(\Omega)\,.
\end{equation}
\end{thm}

Some remarks are in order.

\begin{rem} \label{r:LL-LZ}
\begin{itemize}
\item[(i)] As  in Remark \ref{r:Z}-(i),  the right-hand side of \eqref{est:obs_log} is well-defined and finite, under our hypotheses. However, as the proof is much more involved, we will clarify it in detail in subsection \ref{ss:E}.

\item[(ii)] As we will see in Theorem \ref{th:L-LL}, the loss of a finite number of derivatives in the observability inequality cannot be avoided, i.e. we do have $m>0$.

\item[(iii)] Note that the conditions we impose on the higher order derivatives $D_\omega^mu_0$ and $D_\omega^mu_1$ are very strong. In particular they imply that $D_\omega^ku_0, D_\omega^ku_1\,\in H^2(\Omega)$ for any $0\leq k\leq m-1$.
However, they are necessary, due to the weak modulus of continuity of $\omega$. For instance, let us consider the first case for which they are not trivial, i.e. $m=2$. So, our hypotehses imlpy $u_0\in H^5(\Omega)$, and then $\d_x^2u_0\in H^3$. Nevertheless, from this we cannot infer $D_\omega u_0$ to belong to the same space (hence, a fortiori, neither that $D_\omega^2u_0\in H^1$), due to the rough regularity of $\omega$.
So, having smooth initial data doesn't help to propagate their regularity in the equation.

\item[(iv)] Note that $m$ is linked with the quantity $|\omega|_{\mc{A}}/\omega_*$, with $\mc{A}=LL$ or $LZ$ if $\omega$ is log-Lipschitz or log-Zygmund continuous respectively (see also subsection \ref{ss:cauchy} and Remark \ref{r:loss}), but it's independent of $T$. The observability constant, $C$, however, depends on all these quantities.

\item[(v)] We point out here that the loss comes from Propositions \ref{p:LL} and \ref{p:Z-LZ}, where it's given by the index $\beta>0$. In our statement, $m=[\beta]+1$: we prefer to work with an integer loss, to simplify the presentation and for applications to the control problem. However, according to the just mentioned propositions, a more precise inequality would be
\begin{equation} \label{est:obs_log-b}
\left\|u_0\right\|^2_{H^1_0(\Omega)}\,+\,\left\|u_1\right\|^2_{L^2(\Omega)}\,\leq\,C\,\bigl\|\d_xu(\,\cdot\,,0)\bigr\|_{H^\beta(0,T)}^2\,
\end{equation}
for any initial data $u_0\in H^{\wtilde{m}+1}(\Omega)\cap H^1_0(\Omega)$ and $u_1\in H^{\wtilde{m}}(\Omega)$ (where we set $\wtilde{m}=[2\beta]$) such that, if $\wtilde{m}>1$, the following additional conditions are verified:
\begin{itemize}
\item if $\wtilde{m}=2k$, then $D_\omega^ku_0\in H^1(\Omega)$ and $D_\omega^ku_1\in L^2(\Omega)$;
\item if $\wtilde{m}=2k+1$, then $D_\omega^ku_1\in H^1(\Omega)$ and $D_\omega^ku_0\in L^2(\Omega)$;
\end{itemize}
Note that, for $\wtilde{m}=1$, then the hypothesis on the initial data are enough to recover $u_1\in H^1(\Omega)$ and $D_\omega u_0\in L^2(\Omega)$, so that no further requirements are needed.
\end{itemize}
\end{rem}

The case of the first variation of the coefficient, i.e. Lipschitz-type conditions, is particularly interesting because we are able to prove the sharpness of our results.
More precisely, we will prove also the following facts (see subsection \ref{ss:first} for the rigorous statements):
\begin{enumerate}
\item any intermediate modulus of continuity between the Lipschitz and the log-Lipschitz ones entails an observability estimate analogous to inequality \eqref{est:obs_log-b}, with an arbitrarly small $\beta>0$, which, however, cannot vanish;

\item any modulus of continuity slightly worse than the log-Lipschitz one always entails an infinite loss of derivatives, so that observability estimates always fail: in particular,
\eqref{est:obs_log} does not hold true, independently of the size of the time $T$ and of the order $m\in\N$.
\end{enumerate}
Together with the results in \cite{FC-Z} and \cite{Castro-Z} we mentioned before, these results complete the general picture about how observability estimates depend on the modulus of continuity of the coefficients.

Let us notice that the last two points will be proved  constructing counterexamples, in the same spirit of the one in paper \cite{Castro-Z} by Castro and Zuazua; they are inspired also by the ones (for the Cauchy problem) of Colombini and Lerner (see \cite{C-L}) for moduli worse than log-Lipschitz, and by Cicognani and Colombini (in \cite{Cic-C}) for moduli between Lipschitz and log-Lipschitz regularity. Hence, we will closely follow the structure of the counterexample in \cite{Castro-Z}: the difficulty is to find the correct oscillation of the coefficients in order to reproduce the same energy concentration phenomenon. We will be more precise in subsection \ref{ss:first}, giving the exact statements and all the details of the construction.

Let us point out that, from the previous characterization, one gathers an analogue counterpart for Zygmund type conditions. Nevertheless, explicit counterexamples in this instance (without resorting to the ones established for the first variation) are still far to be found.

\medbreak
The rest of the paper is organized in the following way.

In the next section we will collect some useful properties we will need in our study. In particular, in a first time we will analyse the functional spaces we are dealing with and recall some of their
basic properties. Then, we will mention some well-known results about the Cauchy problem for second order strictly hyperbolic operators with low regularity coefficients. As a matter of fact,
due to the sidewise energy estimates method, our statements will strongly rely on energy estimates for such a kind of operators.

In section \ref{s:proof}, we will prove Theorems \ref{th:Z} and \ref{th:LL-LZ}. The main ingredient will be the sidewise energy estimates, which will make use, in a crucial way, of the just established energy estimates. Note that, even if stated in the case of the whole $\R$, we will be able to use them for a bounded domain thanks to the finite propagation speed issue.

 In section \ref{s:sharp} we will discuss the optimality of our results. In the case of Zygmund type conditions, in a first time we will give some ``empirical'' considerations in favour of it; then, we will formally derive it from the case of the first variation. So, in this instance we will construct the couterexamples we announced before, estblishing in this way a sharp relation between the modulus of continuity of $\omega$ and the loss of derivatives in the observability estimates.

In section \ref{s:control} we will present the application of our main results to the problem of null controllability for the wave equation.

Finally, in section \ref{s:further} we will discuss some other closely related issues.

\subsubsection*{Acknowledgements}

The authors were supported by Grant MTM2011-29306-C02-00, MICINN, Spain, ERC Advanced Grant FP7-246775 NUMERIWAVES, ESF Research Networking Programme OPTPDE and Grant PI2010-04 of the Basque Government.

\section{Tools} \label{s:tools}

The present section is devoted to present the tools and preliminary results we will need in order to prove our statements.

Hence, in a first time, we will introduce the Littlewood-Paley decomposition and the definition of Besov spaces that will be used in the next subsection, in order to analyse the classes of coefficients we deal with, and to state some of their properties. There, we will also focus on Zygmund type conditions in more detail. In particular, we will compare the ``integral Zygmund'' condition and the $BV$ classes.

Finally, we will recall some basic results on the Cauchy problem, and in particular on energy estimates, for a second order strictly hyperbolic operator with non-Lipschitz coefficients. They will be fundamental in the sequel, especially in the application of the sidewise energy estimates method.

\subsection{An overview on Littlewood-Paley theory} \label{ss:LP}

We introduce here the Littlewood-Paley decomposition and, by use of it,  Besov spaces. We will focus on their basic properties, and the definition of Bony's paraproduct decomposition.

This will be very useful to describe the sets of coefficients we are handling, to justify all the computations we made in the introduction in order to reduce $(W\!E)$ to \eqref{eq:we} and also to explain Remark \ref{r:LL-LZ}-(ii).
For the sake of completeness, we have to mention that Littlewood-Paley theory is also the basis to prove all the results we are going to quote in subsection \ref{ss:cauchy}.
Finally, we point out here that dyadic decomposition has been recently used to deal with multi-dimensional observability problems (see for instance papers \cite{D-L} and \cite{D-E}).

For a complete and detailed description of the Littlewood-Paley theory and
of paradifferential calculus, we refer to \cite{B-C-D}, chapter 2 (see also \cite{Bony} and \cite{M-2008}).

\medbreak
Let us first define the so called ``Littlewood-Paley decomposition'', based on a non-homogeneous dyadic partition of unity with
respect to the Fourier variable.
So, fix a smooth radial function
$\chi$ supported in (say) the ball $B(0,4/3),$ 
equal to $1$ in a neighborhood of $B(0,3/4)$
and such that $r\mapsto\chi(r\,e)$ is nonincreasing
over $\R_+$ for all unitary vector $e\in\R^N$. Moreover, set
$\varphi\left(\xi\right)=\chi\left(\xi/2\right)-\chi\left(\xi\right).$
\smallbreak
The dyadic blocks $(\Delta_j)_{j\in\Z}$
 are defined by\footnote{Throughout we agree  that  $f(D)$ stands for 
the pseudo-differential operator $u\mapsto\mc{F}^{-1}(f\,\mc{F}u)$.} 
$$
\Delta_j:=0\ \hbox{ if }\ j\leq-2,\quad\Delta_{-1}:=\chi(D)\quad\hbox{and}\quad
\Delta_j:=\varphi(2^{-j}D)\ \text{ if }\  j\geq0.
$$
We  also introduce the following low frequency cut-off operator:
$$
S_ju:=\chi(2^{-j}D)=\sum_{k\leq j-1}\Delta_{k}\quad\text{for}\quad j\geq0.
$$
The following fundamental properties hold true:
\begin{itemize}
\item for any $u\in\mc{S}',$ the equality $u=\sum_{j}\Delta_ju$ holds true in $\mc{S}'$;
\item for all $u$ and $v$ in $\mc{S}'$,
the sequence $\left(S_{j-1}u\,\,\Delta_jv\right)_{j\in\N}$ is spectrally supported in dyadic annuli,
where the size of the frequencies is proportional to $2^j$.
\end{itemize}

Before going on, let us mention a fundamental result, which explains, by the so-called \emph{Bernstein's inequalities},
the way derivatives act on spectrally localized functions.
  \begin{lemma} \label{l:bern}
Let  $0<r<R$.   A
constant $C$ exists so that, for any nonnegative integer $k$, any couple $(p,q)$ 
in $[1,+\infty]^2$ with  $p\leq q$ 
and any function $u\in L^p$,  we  have, for all $\lambda>0$,
$$
\displaylines{
{\rm supp}\, \widehat u \subset   B(0,\lambda R)\quad
\Longrightarrow\quad
\|\nabla^k u\|_{L^q}\, \leq\,
 C^{k+1}\,\lambda^{k+N\left(\frac{1}{p}-\frac{1}{q}\right)}\,\|u\|_{L^p}\;;\cr
{\rm supp}\, \widehat u \subset \{\xi\in\R^N\,|\, r\lambda\leq|\xi|\leq R\lambda\}
\quad\Longrightarrow\quad C^{-k-1}\,\lambda^k\|u\|_{L^p}\,
\leq\,
\|\nabla^k u\|_{L^p}\,
\leq\,
C^{k+1} \, \lambda^k\|u\|_{L^p}\,.
}$$
\end{lemma}   

One can now define what a (non-homogeneous) Besov space $B^s_{p,r}$ is.
\begin{defin} \label{d:besov}
  Let  $u$ be a tempered distribution, $s$ a real number, and 
$1\leq p,r\leq+\infty.$ We define the space $B^s_{p,r}$ as the set of distributions $u\in\mc{S}'$ such  that
$$
\|u\|_{B^s_{p,r}}:=\bigg\|\bigl(2^{js}\,
\|\Delta_j  u\|_{L^p}\bigr)_{j\geq-1}\bigg\|_{\ell^r}\;<\;+\infty\,.
$$
\end{defin}
  
From the above definition, it is easy to show that for all $s\in\R$, the Besov space $B^s_{2,2}$ coincides with the non-homogeneous
Sobolev space $H^s$.

On the other side, for all $s\in\,\R_+\!\!\setminus\!\N$, the space $B^s_{\infty,\infty}$ is actually the
H\"older space $\mc{C}^s$.
If $s\in\N$, instead, we set $\mc{C}^s_*:=B^s_{\infty,\infty}$, to distinguish it from the space $\mc{C}^s$ of
the differentiable functions with continuous partial derivatives up to the order $s$. Moreover, the strict inclusion
$\mc{C}^s_b\,\hra\,\mc{C}^s_*$ holds, where $\mc{C}^s_b$ denotes the subset of $\mc{C}^s$ functions bounded with all
their derivatives up to the order $s$.
For the sake of completeness, let us recall that, if $s<0$, we define the ``negative H\"older space'' $\mc{C}^s$ as the Besov space $B^s_{\infty,\infty}$.

  Finally, let us also point out that for any $k\in\N$ and $p\in[1,+\infty]$, we have the following chain of continuous embeddings:
 $$
 B^k_{p,1}\hookrightarrow W^{k,p}\hookrightarrow B^k_{p,\infty}\,,
 $$
  where  $W^{k,p}$ denotes the classical Sobolev space of $L^p$ functions
 with all the derivatives up to the order $k$ in $L^p$.

\medbreak
Let us recall now some basic facts about paradifferential calculus, as introduced by J.-M. Bony in \cite{Bony}. Again, one can refer also to \cite{B-C-D} and \cite{M-2008}.

Given two tempered distributions $u$ and $v$, formally one has $u\,v\,=\,\sum_{j,k}\Delta_ju\,\Delta_kv$. Now, due to the
spectral localization of cut-off operators, we can write the following \emph{Bony's decomposition}:
\begin{equation}\label{eq:bony}
u\,v\,=\,T_uv\,+\,T_vu\,+\,R(u,v)\,,
\end{equation}
where we have defined the paraproduct and remainder operators respectively as
$$
T_uv\,:=\,\sum_jS_{j-1}u\,\Delta_jv\qquad\hbox{ and }\qquad
R(u,v)\,:=\,\sum_j\,\sum_{|k-j|\leq1}\Delta_ju\,\Delta_kv\,.
$$

Paraproduct and remainder operators enjoy some continuity properties on the class of non-homogeneous Besov spaces.
\begin{thm}\label{t:op}
\begin{itemize}
\item[(i)] For any $(s,p,r)\in\R\times[1,+\infty]^2$ and $t>0,$ the paraproduct operator 
$T$ maps $L^\infty\times B^s_{p,r}$ in $B^s_{p,r}$,
and  $B^{-t}_{\infty,r_1}\times B^s_{p,r_2}$ in $B^{s-t}_{p,q}$, with $1/q\,:=\,\min\left\{1\,,\,1/r_1\,+\,1/r_2\right\}$.
Moreover, the following estimates hold true:
$$
\|T_uv\|_{B^s_{p,r}}\,\leq\,C\,\|u\|_{L^\infty}\|\nabla v\|_{B^{s-1}_{p,r}}\qquad\quad\hbox{and}\qquad\quad
\|T_uv\|_{B^{s-t}_{p,q}}\,\leq\,C\,\|u\|_{B^{-t}_{\infty,r_1}}\|\nabla v\|_{B^{s-1}_{p,r_2}}\,.
$$
\item[(ii)] For any $(s_1,p_1,r_1)$ and $(s_2,p_2,r_2)$ in $\R\times[1,\infty]^2$ such that 
$s_1+s_2\geq0,$ $1/p:=1/p_1+1/p_2\leq1$ and $1/r:=1/r_1+1/r_2\leq1$
the remainder operator $R$ maps 
$B^{s_1}_{p_1,r_1}\times B^{s_2}_{p_2,r_2}$ in $B^{s_1+s_2}_{p,r}$, and one has:
\begin{eqnarray*}
 \left\|R(u,v)\right\|_{B^{s_1+s_2}_{p,r}}\;\leq\;\frac{C^{s_1+s_2+1}}{s_1+s_2}\,
\|u\|_{B^{s_1}_{p_1,r_1}}\,\|v\|_{B^{s_2}_{p_2,r_2}} & \quad\hbox{ if } & s_1+s_2\,>\,0 \\ [1ex]
 \left\|R(u,v)\right\|_{B^0_{p,\infty}}\;\leq\;C^{s_1+s_2+1}\,
\|u\|_{B^{s_1}_{p_1,r_1}}\,\|v\|_{B^{s_2}_{p_2,r_2}} & \quad\hbox{ if } & s_1+s_2\,=\,0\,,\;\;r=1\,.
\end{eqnarray*}
\end{itemize}
\end{thm}



Combining Theorem \ref{t:op} with Bony's paraproduct decomposition \eqref{eq:bony}, 
we easily get the following ``tame estimate''.
\begin{coroll}\label{c:op}
Let $a$ be a bounded function such that $\nabla a\in B^{s-1}_{p,r}$ for some $s>0$
and $(p,r)\in[1,+\infty]^2.$  Then for any $b\in B^s_{p,r}\cap L^\infty$ we have $ab\in B^s_{p,r}\cap L^\infty$
and there exists a constant $C$, depending only on $N,$ $p$ and $s$, such that 
$$
\|ab\|_{B^s_{p,r}}\leq C\Bigl(\|a\|_{L^\infty}\|b\|_{B^s_{p,r}}+\|b\|_{L^\infty}\|\nabla a\|_{B^{s-1}_{p,r}}\Bigr).
$$
\end{coroll}

Let us also recall the action of composition by smooth functions on Besov spaces. First of all, we have the following result.
\begin{thm} \label{t:comp}
 Let $f\in\mc{C}^\infty(\R)$ such that $f(0)=0$, $s>0$ and $(p,r)\in[1,+\infty]^2$.

If $u\in L^\infty\cap B^s_{p,r}$, then so does $f\circ u$ and moreover
$$
\left\|f\circ u\right\|_{B^s_{p,r}}\,\leq\,C\,\left\|u\right\|_{B^s_{p,r}}\,,
$$
for a constant $C$ depending only on $s$, $f'$ and $\|u\|_{L^\infty}$.
\end{thm}
We can state another result (see paper \cite{D-2010}, section 2, for its proof), which is strictly related to the previous one.
\begin{prop} \label{p:comp}
 Let $I\subset\R$ be an open interval and $f:I\,\longrightarrow\,\R$ be a smooth function.

Then, for all compact subset $J\subset I$, $s>0$ and $(p,r)\in[1,+\infty]^2$, there exists a constant $C$ such that,
for all functions $u$ valued in $J$ and with gradient $\nabla u\in B^{s-1}_{p,r}$, we have that also
$\nabla(f\circ u)\in B^{s-1}_{p,r}$ and
$$
\left\|\nabla\left(f\circ u\right)\right\|_{B^{s-1}_{p,r}}\,\leq\,C\,\left\|\nabla u\right\|_{B^{s-1}_{p,r}}\,.
$$
\end{prop}

\subsection{The space $\mc{A}$ of the coefficients} \label{ss:coeff}

Let us now introduce the functional spaces we want to deal with.
For the time being let us deal with pointwise conditions, which are in some sense classical. We will generalize them to integral ones in the next subsection.

Modulo the extension of  the coefficients, we can suppose them to be globally defined over $\R_x$. Note here that the pointwise conditions imply, in particular, $\omega$ to be continuous; then, we extend it by the constant values $\omega(0)$ for $x<0$ and $\omega(1)$ for $x>0$. Under integral assumptions, instead, the continuity is lost, and then we can extend just by $0$ out of the domain $\Omega$.

Finally, for the sake of generality we will consider the instance of any space dimension $N\geq1$.

\begin{defin} \label{d:LL}
 A function $f\in L^\infty(\R^N)$ is said to be \emph{log-Lipschitz} continuous, and we write $f\in LL(\R^N)$, if the quantity
$$
|f|_{LL,\infty}\,:=\,\sup_{x,y\in\R^N,\,|y|<1}
\left(\frac{\left|f(x+y)\,-\,f(x)\right|}{|y|\,\log\left(1\,+\,\frac{1}{|y|}\right)}\right)\,<\,+\infty\,.
$$
We set $\|f\|_{LL}\,:=\,\|f\|_{L^\infty}\,+\,|f|_{LL,\infty}$.
\end{defin}
Let us define also some Zygmund classes.
\begin{defin} \label{d:LZ}
 A function $g\in L^\infty(\R^N)$ is said to be \emph{log-Zygmund} continuous, and we write $g\in LZ(\R^N)$, if the quantity
$$
|g|_{LZ,\infty}\,:=\,\sup_{x,y\in\R^N,\,|y|<1}
\left(\frac{\left|g(x+y)\,+\,g(x-y)\,-\,2\,g(x)\right|}{|y|\,\log\left(1\,+\,\frac{1}{|y|}\right)}\right)\,<\,+\infty\,.
$$
We set $\|g\|_{LZ}\,:=\,\|g\|_{L^\infty}\,+\,|g|_{LZ,\infty}$.

The space $Z(\R^N)$ of \emph{Zygmund} continuous functions is defined instead by the condition
$$
|g|_{Z,\infty}\,:=\,\sup_{x,y\in\R^N,\,|y|<1}
\left(\frac{\left|g(x+y)\,+\,g(x-y)\,-\,2\,g(x)\right|}{|y|}\right)\,<\,+\infty\,,
$$
and, analogously, we set $\|g\|_{Z}\,:=\,\|g\|_{L^\infty}\,+\,|g|_{Z,\infty}$.
\end{defin}

Let us recall that $Z\equiv B^1_{\infty,\infty}$ (see e.g. \cite{Ch1995} for the proof), while the space $LZ$ coincides with the logarithmic Besov space $B^{1-\log}_{\infty,\infty}$ (see for instance \cite{C-DS-F-M_th}, section 3), which is defined by the condition
$$
\left\|u\right\|_{B^{1-\log}_{\infty,\infty}}\,:=\,\sup_{\nu\geq-1}\left(2^\nu\,\left(\nu+1\right)^{-1}\,\left\|\Delta_\nu u\right\|_{L^\infty}\right)\,<\,+\infty\,.
$$
Logarithmic Sobolev spaces were first introduced in \cite{C-M}: they come out in a natural way in the study of wave equations with non-Lipschitz coefficients.
For the generalization to the class of logarithmic Besov spaces, one can refer to \cite{F_phd}. However, it is
 enough to keep in mind that they are intermediate classes between the classical ones, of which they enjoy the most of the properties (the proofs can be obtained just slightly modifying the classical arguments).

We do not have an exact identification of Lipschitz-type classes as Besov spaces; however we can characterize the space $LL$ by the condition (on the low frequencies)
$$
\sup_{\nu\geq-1}\biggl((\nu+1)^{-1}\,\left\|\nabla S_\nu u\right\|_{L^\infty}\biggr)\,<\,+\infty
$$
(see paper \cite{C-L}, Prop. 3.3).

Now, it is
 evident that $LL$ is an algebra, and that it is
 invariant under the transformation $\mc{I}:z\mapsto1/z$ for functions which fulfill \eqref{eq:hyp}.
Furthermore, by Bony's paraproduct decomposition and their dyadic characterization, it turns out that both $Z$ and $LZ$ are algebras; the fact that they are still invariant under the action of $\mc{I}$ follows from Proposition \ref{p:comp} and hypothesis \eqref{eq:hyp}.

Finally, let us recall the following embeddings: $Z\hra LL\hra LZ$. The latter is evident from the definitions, while the former can be proved thanks again to the Littlewood-Paley decomposition (see e.g. \cite{B-C-D}, chap. 2).

\subsection{On the Zygmund condition} \label{ss:zygmund}

As mentioned in the introduction, we defined the class $\mc{A}$ of the coefficients by integral conditions. A fortiori, Theorems \ref{th:Z} and \ref{th:LL-LZ} hold true under the pointwise assumptions
\begin{eqnarray}
\sup_{x\in\R}\,\bigl|\omega(x+y)\,-\,\omega(x)\bigr| & \leq & C\,|y|\,\log\left(1+\frac{1}{|y|}\right) \label{est:LL} \\
\sup_{x\in\R}\,\bigl|\omega(x+y)\,+\,\omega(x-y)\,-\,2\,\omega(x)\bigr| & \leq & C\,|y| \label{est:Z}
\end{eqnarray}
(and their analogue for the log-Zygmund behaviour), which we introduced in the previous subsection.

So, it is
 natural to generalize the Zygmund classes in the following sense. Again, we present the subject in the general instance of $\R^N$.
\begin{defin} \label{def:zyg_p}
Let $p\in[1,+\infty]$. We define the space $\mc{Z}_p(\R^N)$ as the set of $f\in L^p(\R^N)$ such that there exists a constant $C>0$ for which
$$
\bigl\|f(\,\cdot\,+y)\,+\,f(\,\cdot\,-y)\,-\,2\,f(\,\cdot\,)\bigr\|_{L^p(\R^N)}\,\leq\,C\,|y|
$$
for all $y\in\R^N$ with $|y|<1$.
\end{defin}
Note that $Z$, introduced in Definition \ref{d:LZ}, coincides with $\mc{Z}_\infty$, while the integral condition \eqref{est:Z_1} defines the space $\mc{Z}_1$.

Note also that one can define the corresponding spaces $\mc{LZ}_p$ introducing a logarithmic loss in the right hand side of the previous relation. Obviously, nothing changes with respect to this case.

Exactly as for the $L^\infty$ instance, the following proposition holds true.
\begin{prop} \label{p:zygm_p}
For any $p\in[1,+\infty]$, the classes $\mc{Z}_p(\R^N)$ and $\mc{LZ}_p(\R^N)$ coincide, respectively, with the Besov spaces $B^1_{p,\infty}(\R^N)$ and $B^{1-\log}_{p,\infty}(\R^N)$.
\end{prop}

\begin{proof}
We will closely follow the lines of the classical proof (for $p=+\infty$). Let us just focus on the $\mc{Z}_p$ instance: the logarithmic loss can be handled in an analogous way.
\begin{itemize}
  \item[(i)] Let us first consider a $f\in B^{1}_{p,\infty}$ and take $x$ and $y\,\in\R^N$, with
$|y|<1$. For all fixed $n\in\N$ we can write:
\begin{eqnarray*}
f(x+y)+f(x-y)-2f(x) & = & \sum_{k<n}\left(\Delta_kf(x+y)+\Delta_kf(x-y)-2\Delta_kf(x)\right)\,+ \\
& & \qquad\qquad+\,\sum_{k\geq n}\left(\Delta_kf(x+y)+\Delta_kf(x-y)-2\Delta_kf(x)\right)\,.
\end{eqnarray*}
First, we take advantage of the Taylor's formula up to second order to handle the former terms; then, we take the $L^p$ norms of both sides. Using also Definition \ref{d:besov}
and Bernstein's inequalities, we get
\begin{eqnarray*}
 \left\|f(x+y)+f(x-y)-2f(x)\right\|_{L^p_x} & \leq & C\,|y|^2\sum_{k<n}\left\|\nabla^2\Delta_kf\right\|_{L^p_x}\,+\,
4\,\sum_{k\geq n}\left\|\Delta_kf\right\|_{L^p_x} \\
& \leq & C\left(|y|^2\sum_{k<n}2^k\,+\,\sum_{k\geq n}2^{-k}\right) \\
& \leq & C\,\left(|y|^2\,2^n\,+\,2^{-n}\right)\,.
\end{eqnarray*}
Now, as $|y|<1$, the choice $n=1+\left[\log_2\left(1/|y|\right)\right]$ (where with $[\sigma]$ we mean the greatest positive
integer less than or equal to $\sigma$) completes the proof of the first part.
\item[(ii)] Now, given a function $f\in\mc{Z}_p$, we want to estimate the $L^\infty$ norm of its localized part $\Delta_kf$.

Notice that applying the operator $\Delta_k$ plays the same role as the convolution with the inverse Fourier transform of
the function $\vphi(2^{-k}\cdot)$, which we call $h_k(x)=2^{kN}h(2^k\cdot)$, where we set
$h=\mc{F}^{-1}_\xi(\vphi)$. As $\vphi$ is an even function, so does $h$; moreover we have
$$
\int h(z)\,dz\,=\,\int\mc{F}^{-1}_\xi(\vphi)(z)\,dz\,=\,\vphi(\xi)_{|\xi=0}\,=\,0\,.
$$
Therefore, we can write:
$$
\Delta_kf(x)\,=\,2^{kN-1}\int h(2^ky)\left(f(x+y)+f(x-y)-2f(x)\right)\,dy\,,
$$
and using the definition of the space $\mc{Z}_p$ completes the proof of the second part.
 \end{itemize}

The proposition is now completely proved.
\end{proof}

\begin{coroll} \label{c:Z-BV}
$\mc{Z}_1\,\equiv\,B^1_{1,\infty}$. In particular, $W^{1,1}\,\hra\,\mc{Z}_1$.
\end{coroll}

\begin{ex} \label{ex:Z}
In \cite{Tar}  an example is given of a  $\mc{Z}_1(0,2\pi)$ function $w$, but for which there is no  constant $C$ such that, for all $0<h<1$, one has
$$
\int_{0}^{2\pi-h}\bigl|w(x+h)\,-\,w(x)\bigr|\,dx\,\leq\,C\,h\,.
$$

Recall  (see \cite{Brezis}, chapter 8) that the space $BV(\Omega)$ (in any dimension $N\geq1$) is characterized by the following property:
\textit{there exists a constant $C>0$ such that, for all $\Theta\subset\subset\Omega$ and all $|h|<\mbox{dist}(\Theta,\Omega)$, one has
$$
\int_\Theta\bigl|w(x+h)\,-\,w(x)\bigr|\,dx\,\leq\,C\,|h|\,.
$$}

After a simple rescaling of Tarama's example, we then have that the Weierstrass function
$$
w(x)\,=\,\sum_{n=1}^{+\infty}2^{-n}\,\cos\bigl(2^{n+1}\,\pi\,x\bigr)
$$
is in the class $\mc{Z}_1(\Omega)$, but  not in $BV(\Omega)$.
\end{ex}

Therefore, from this point of view, Theorem \ref{th:Z} represents an extension of the result of \cite{FC-Z} for $BV$ coefficients.
However Zygmund conditions are much more related with the second derivative (for  smooth functions). 
\subsection{Energy estimates for hyperbolic operators with rough coefficients} \label{ss:cauchy}

Due to the sidewise energy estimates technique, which we will explain in the next section, our results strongly rely on energy estimates for second order strictly hyperbolic operators
with non-Lipschitz coefficients
\begin{equation} \label{def:L}
Lv(\tau,y)\,:=\,\d_\tau^2v(\tau,y)\,-\,\omega(\tau)\,\d_y^2v(\tau,y)\,,
\end{equation}
defined in some strip $[0,\mc{T}]\times\R_y$. Let us point out that we are considering the $1$-dimensional case, which is enough for our purposes, even if all the facts we are going to quote in the present subsection are true
also in the general instance $\R^N_y$, with $N\geq1$.

Here $\omega$ enjoys property \eqref{eq:hyp} as before, and belongs to the space $\mc{A}(\R_\tau)$. The class $\mc{A}$ can be, as before, $LL$, $Z$ or $LZ$ (defined either pointwise or by an integral condition). Note that, under these assumptions,
Theorem  4 of \cite{C-DG-S} still holds true: from it we gather the finite propagation speed for operator \eqref{def:L}. Thanks to this, we can then apply the results we are going to quote to  the case of bounded domains under consideration.

\medbreak
It is
 well-known that energy estimates are relevant to prove the well-posedness of the Cauchy problem for the operator $L$ on Sobolev spaces. Such a question for operator \eqref{def:L} with non-Lipschitz coefficients has been studied for a long time (see for instance papers \cite{H-S} and \cite{C-DG-S}), and there is a quite broad literature devoted to it; nevertheless, in despite of this, it is
 still not completely well-understood.

Just to have an idea, keep in mind that, under a Lipschitz continuity condition on the coefficient $\omega$, an energy estimate with no loss of derivatives holds true (in suitable Sobolev norms).
Whenever the Lipschitz condition fails, instead, one can find, in general, energy estimates with (finite or infinite) loss of derivatives: the regularity of the solution deteriorates as the time goes on, and we can control just worse Sobolev norms than the ones of the initial data.

In what follows, we limit ourselves to quote just the results we will need to prove our statements. For an overview of the problem and the present state of the art, we refer e.g. to \cite{C-DS-F-M_wp}
and the references therein.

\medbreak
So, let us start recalling a result on the log-Lipschitz continuity condition, whose proof  can be found in \cite{C-DG-S}. We refer to paper \cite{C-L} for the case of  log-Lipschitz continuous coefficients $\omega=\omega(\tau,y)$ on all its variables.

\begin{prop} \label{p:LL}
Let us consider the operator $L$, defined by \eqref{def:L}, with $\omega$ fulfilling condition  \eqref{est:LL_1} and such that $0<\omega_*\leq\omega(\tau)\leq\omega^*$.

For any $s\in\R$, there exist positive constants $\beta$ (depending only on $\omega_*$ and $\omega^*$ and on $|\omega|_{LL}$) and $C$ (depending also on $s$, but not on $|\omega|_{LL}$), such that for $\tau\in[0,\mc{T}]$ one has
\begin{equation} \label{est:en_LL} 
\left\|v(\tau)\right\|_{H^{s-\beta}}+\left\|\d_\tau v(\tau)\right\|_{H^{s-1-\beta}}\,\leq\,C\left(\left\|v(0)\right\|_{H^{s}}+\left\|\d_\tau v(0)\right\|_{H^{s-1}}+\int^\tau_0\left\|Lv(\tau')\right\|_{H^{s-1-\beta}}d\tau'\right)
\end{equation}
for all $v\in\mc{C}^2([0,\mc{T}];H^\infty(\R_y))$.
\end{prop}

Let us also recall another result, whose proof can be found in paper \cite{Cic-C}. It stems that any intermediate modulus of continuity between the Lipschitz and the log-Lipschitz ones always entails a loss of derivatives in the energy estimates. The necessity of such a loss was proved by construction of a counterexample: we will come back to it in section \ref{s:sharp}.
\begin{prop} \label{p:L-LL}
Let $\mu:[0,1]\ra[0,1]$ continuous and strictly increasing, with $\mu(0)=0$, and suppose that $\omega$ fulfills
$$
\sup_{\tau\in\R}\bigl|\omega(\tau+\sigma)\,-\,\omega(\tau)\bigr|\,\leq\,K\,|\sigma|\,\log\left(1+\frac{1}{|\sigma|}\right)\,\mu(|\sigma|)
$$
for some constant $K>0$, for any $|\sigma|<1$.

Then, for any $\delta>0$ and any $s\in\R$, there is a constant $C$ (depending on $s$, $\delta$ and on $\omega_*$ and $\omega^*$), such that
\begin{equation} \label{est:en_L-LL} 
\left\|v(\tau)\right\|_{H^{s-\delta}}+\left\|\d_\tau v(\tau)\right\|_{H^{s-1-\delta}}\,\leq\,C\left(\left\|v(0)\right\|_{H^{s}}+\left\|\d_\tau v(0)\right\|_{H^{s-1}}+\int^\tau_0\left\|Lv(\tau')\right\|_{H^{s-1-\delta}}d\tau'\right)
\end{equation}
for any $\tau\in[0,\mc{T}]$ and for all $v\in\mc{C}^2([0,\mc{T}];H^\infty(\R_y))$.
\end{prop}

Now, let us focus on the second variation of the coefficients. The following result was stated in \cite{Tar} for integral conditions. See also works \cite{C-DS}, \cite{C-DS-F-M_th} and
\cite{C-DS-F-M_wp} for some generalizations.

\begin{prop} \label{p:Z-LZ}
Fix $s\in\R$, and let $v\in\mc{C}^2([0,\mc{T}];H^\infty(\R_y))$.
\begin{itemize}
\item[(i)] If $\omega$ verifies the integral inequality \eqref{est:Z_1}, then
\begin{equation} \label{est:en_Z} 
\left\|v(\tau)\right\|_{H^{s}}\,+\,\left\|\d_\tau v(\tau)\right\|_{H^{s-1}}\,\leq\,C\left(\left\|v(0)\right\|_{H^s}\,+\,\left\|\d_\tau v(0)\right\|_{H^{s-1}}\,+\,\int^\tau_0\left\|Lv(\tau')\right\|_{H^{s}}\,d\tau'\right),
\end{equation}
for a constant $C$ which depends only on $s$, $\omega_*$, $\omega^*$ and $|\omega|_Z$.
\item[(ii)] If $\omega$ fulfills instead \eqref{est:LZ_1}, then an estimate analogous to \eqref{est:LL} holds true, for positive constants $\beta$ (depending only on $\omega_*$, $\omega^*$ and on $|\omega|_{LZ}$) and $C$ (depending also on $s$, but not on $|\omega|_{LZ}$).
\end{itemize}
\end{prop}

We expect an analogous statement to hold true also for intermediate conditions between the Zygmund and the log-Zygmund one. However, this goes beyond the scope of the present paper, and it will be matter of next studies.

\begin{rem} \label{r:loss}
Going along the lines of the proofs to Propositions \ref{p:LL} and \ref{p:Z-LZ} (and assuming, without loss of generality, that $\omega_*<1$), it is
 possible to see that the constant $\beta$ occurring in \eqref{est:en_LL} is, respectively,
$$
\beta_{LL}\,=\,\wtilde{\beta}\;|\omega|_{LL}\;\frac{1}{\omega_*}\qquad\mbox{ and }\qquad
\beta_{LZ}\,=\,\wtilde{\beta}\;|\omega|_{LZ}\;\frac{1}{\omega^2_*}\,,
$$
for some ``universal'' constant $\wtilde{\beta}>0$.
\end{rem}

\section{Proof of the results} \label{s:proof}

Let us now start the proof of our main results. It will be based on genuinely one-dimensional arguments, following the main ideas of \cite{FC-Z} (see also \cite{Z-1993}): in particular, we will apply the \emph{sidewise energy estimates} technique. It consists in changing the role of the variables, seeing $x$ as the new evolution parameter, and in establishing bounds for the energy associated to the ``new'' equation.

We will focus on the proof of Theorem \ref{th:LL-LZ}. Indeed, the proof of Theorem \ref{th:Z} can be obtained with analogous computations: so, we will just point out the main features in order to get it.

We remark once again that it is possible to use the results of subsection \ref{ss:cauchy} due to the finite propagation speed issue (proved in \cite{C-DG-S} under more general assumptions).


\subsection{Time evolution}  \label{ss:E}

As, under both hypotheses of Theorems \ref{th:Z} and \ref{th:LL-LZ}, $\omega$ is measurable and bounded and we have finite propagation speed, we get (see \cite{H-S}) the existence and uniqueness of a solution $u$ to the problem \eqref{eq:we}, such that
$$
u\,\in\,\mc{C}([0,T];H^1_0(\Omega))\,\cap\,\mc{C}^1([0,T];L^2(\Omega))
$$
and for which we have the conservation of the ``classical'' energy:
$$
E_0(t)\,:=\,\frac{1}{2}\,\int_\Omega\left(\omega(x)\,\left|\d_tu(t,x)\right|^2\,+\,\left|\d_xu(t,x)\right|^2\right)dx\,.
$$
As a matter of fact, an easy computation gives $E'_0(t)\equiv0$.

Actually, due to the time reversibility of our equation, we can consider also the evolution for negative times, for which we still have conservation of energy. In particular, there exists a constant $C>0$ such that, for any $T^*>0$, then
\begin{equation} \label{est:past-future}
\sup_{[-T^*,T^*]}\left(\int_\Omega\left(\omega(x)\,\left|\d_tu(t,x)\right|^2\,+\,\left|\d_xu(t,x)\right|^2\right)dx\right)\,\leq\,C\,\left(\|u_0\|^2_{H^1_0(\Omega)}\,+\,\|u_1\|^2_{L^2(\Omega)}\right).
\end{equation}

Let us now focus on the integral Zygmund instance for a while. Under the hypothesis of Theorem \ref{th:Z}, as explained in Remark \ref{r:Z}-(i), it makes sense to consider the trace of $\d_xu$ at any point $x\in\Omega$.

Then, with the terminology we will introduce in the next subsection, we can consider the ``sidewise'' problem: we invert equation \eqref{eq:we} and we look at the evolution with respect to $x$. Thanks to space reversibility, applying Proposition \ref{p:Z-LZ} with the slice $(\d_tu,\d_xu)_{|x}$ as new initial data (and using again finite propagation speed), for any $x\in\Omega$ we find
\begin{eqnarray}
& & \sup_{y\in[0,x]}\left(\int_{-T_\omega y}^{T+T_\omega y}\left(\left|\d_tu(t,y)\right|^2\,+\,\left|\d_xu(t,y)\right|^2\right)dt\right)\,\leq \label{est:sidewise_en} \\
& & \qquad\qquad\qquad\qquad\qquad\qquad\qquad\qquad
\leq\,C\,\int_{-T_\omega x}^{T+T_\omega x}\left(\omega(x)\left|\d_tu(t,x)\right|^2\,+\,\left|\d_xu(t,x)\right|^2\right)dt\,. \nonumber
\end{eqnarray}
In particular, the previous relation is true when we compute the left-hand side at $y=0$:
$$
\int_{0}^{T}\left|\d_xu(t,0)\right|^2\,dt\,\leq\,C\,\int_{-T_\omega x}^{T+T_\omega x}\left(\omega(x)\left|\d_tu(t,x)\right|^2\,+\,\left|\d_xu(t,x)\right|^2\right)dt
$$
for any $x\in\Omega$.
Note that, a priori, the right-hand side might be equal to $+\infty$. However, by \eqref{est:sidewise_en} we see that either it is $+\infty$ on an interval $[x_0,1]$, for some $0<x_0<1$, or just in the extreme point $x=1$. The former case is excluded by what we are going to say in a while, and reversing this argument (and essentially thanks to the freedom in choosing $T^*$ in \eqref{est:past-future} ) shows that also the latter is impossible.

So, let us integrate the last relation with respect to $x$: we get
\begin{equation} \label{est:boundary_en}
\int_{0}^{T}\left|\d_xu(t,0)\right|^2\,dt\,\leq\,C\,\int_{-T_\omega}^{T+T_\omega}E(t)\,dt\,\leq\,C\,(T+2T_\omega)\,\left(\|u_0\|^2_{H^1_0(\Omega)}\,+\,\|u_1\|^2_{L^2(\Omega)}\right)\,,
\end{equation}
which tells us also that the right-hand side of \eqref{est:obs_Z} is finite.

\medbreak
We now focus on the case when $\omega$ is in the integral log-Lipschitz or log-Zygmund class, as in hypothesis of Theorem \ref{th:LL-LZ}.
In this instance, we need to deal with higher order time derivatives.

Let us start with a lemma.
\begin{lemma} \label{l:en_E}
For all integer $k\geq0$, define the quantity
\begin{equation} \label{eq:E_s}
E_k(t)\,:=\,\frac{1}{2}\,\int_\Omega\left(\omega(x)\,\left|\d^{k+1}_tu(t,x)\right|^2\,+\,\left|\d^k_t\d_xu(t,x)\right|^2\right)dx\,.
\end{equation}

Then $E_k$ is conserved in the time evolution: for all $t\in[0,T]$,
\begin{equation} \label{eq:conserv_E}
E_k(t)\,\equiv\,E_k(0)\,=\,\frac{1}{2}\,\int_\Omega\left(\omega(x)\,\left|\d_t^{k+1}u(0,x)\right|^2\,+\,\left|\d_t^k\d_xu(0,x)\right|^2\right)dx\,.
\end{equation}
\end{lemma}

\begin{proof}
We apply operator $\d_t^k$ to \eqref{eq:we} and we get the equation
$$
\d^{k+2}_tu\,-\,\omega(x)\,\d_x\d_t^{k}u\,=\,0\,.
$$
Note that the null boundary conditions are still fulfilled. Now, it is
 just matter of computing $E_k'(t)$ and use the previous relation.
\end{proof}

Notice that, thanks to \eqref{eq:hyp}, we have
$$
E_k(t)\,\sim\,\left\|\d_t^{k+1}u(t)\right\|^2_{L^2(\Omega)}\,+\,\left\|\d_t^k\d_xu(t)\right\|^2_{L^2(\Omega)}\,.
$$

For simplicity, for the time being, let us restrict to $m=1$ in Theorem \ref{th:LL-LZ}, which corresponds to the case $0<\beta<1$ in Propositions \ref{p:LL} or \ref{p:Z-LZ}, or to the case of Proposition \ref{p:L-LL}.
If we define $v_1:=\d_tu$, then it satisfies the following system:
\begin{equation} \label{eq:v_1}
\left\{\begin{array}{l}
\omega(x)\,\d_t^2v_1\,-\,\d_x^2v_1\,=\,0 \\[1ex]
v_{1\,|t=0}\,=\,u_1\,,\quad \d_tv_{1\,|t=0}\,=\,D_\omega u_0\,.
\end{array}
\right.
\end{equation}
By hypothesis, we have $v_{1\,|t=0}\in H^1(\Omega)$ and $\d_tv_{1\,|t=0}\in L^2(\Omega)$; then
$$
v_1\,\in\,L^\infty([0,T];H^1(\Omega))\,\cap\,W^{1,\infty}([0,T];L^2(\Omega))\,.
$$
From this we infer (as done in the case without loss) that the trace of $\d_xv_1=\d_t\d_xu$ at any point $x\in\Omega$ is well defined. Moreover, by Propositions \ref{p:LL} or \ref{p:Z-LZ} and arguing as before, we find
$$
\int^T_0\bigl|\d_xv_1(t,0)\bigr|^2\,dt\,\leq\,C\,\int_{-T_\omega x}^{T+T_\omega x}\left(\omega(x)\,\bigl|\d_t^2v_1(t,x)\bigr|^2\,+\,\bigl|\d_t\d_xv_1(t,x)\bigr|^2\right)dt
$$
for any $x\in\Omega$. Integrating over $\Omega$ this last relation and using Lemma \ref{l:en_E} above with $k=2$, we end up with the inequality
$$
\int^T_0\bigl|\d_t\d_xu(t,0)\bigr|^2\,dt\,\leq\,C(T,T_\omega)\,\int_{\Omega}\left(\omega(x)\,\bigl|\d_t^2v_1(0,x)\bigr|^2\,+\,\bigl|\d_t\d_xv_1(0,x)\bigr|^2\right)dx\,.
$$
This relation tells us that the right-hand side of observability estimate \eqref{est:obs_log} (with $m=1$) is finite. As a matter of fact, it is enough to notice that \eqref{eq:v_1} implies
$$
\d_t^2v_{1\,|t=0}\,=\,D_\omega u_1\qquad\mbox{ and }\qquad
\d_t\d_xv_{1\,|t=0}\,=\,\d_xD_\omega u_0\,,
$$
and these quantities belong to $L^2(\Omega)$ due to our hypothesis.

In the general instance $m>1$ it is just a matter of arguing by induction.

So, let us define $v_k:=\d_t^ku$ for any $k\geq0$. From system \eqref{eq:v_1} and the one for $v_2$,
$$
\left\{\begin{array}{l}
\omega(x)\,\d_t^2v_2\,-\,\d_x^2v_2\,=\,0 \\[1ex]
v_{2\,|t=0}\,=\,D_\omega u_0\,,\quad \d_tv_{2\,|t=0}\,=\,D_\omega u_1\,,
\end{array}
\right.
$$
it easily follows by induction that, for any $k\geq0$, the odd time derivatives fulfill
$$
\left\{\begin{array}{l}
\omega(x)\,\d_t^2v_{2k+1}\,-\,\d_x^2v_{2k+1}\,=\,0 \\[1ex]
v_{2k+1\,|t=0}\,=\,D_\omega^{k}u_1\,,\quad \d_tv_{2k+1\,|t=0}\,=\,D_\omega^{k+1} u_0
\end{array}\right.\,,
$$
while the even time derivatives satisfy the system
$$
\left\{\begin{array}{l}
\omega(x)\,\d_t^2v_{2k}\,-\,\d_x^2v_{2k}\,=\,0 \\[1ex]
v_{2k\,|t=0}\,=\,D_\omega^{k}u_0\,,\quad \d_tv_{2k+1\,|t=0}\,=\,D_\omega^{k} u_1\,.
\end{array}\right.
$$

Now, in order to prove that the quantity in the right-hand side of inequality \eqref{est:obs_log} is well-defined, we have to analyse the equation for $v_m$, which will be one of the previous two ones depending if $m$ is either odd or even.
Thanks to the assumptions of Theorem \ref{th:LL-LZ} (recall in particular Remark \ref{r:LL-LZ}-(iii) ), we find that
$$
v_m\,=\,\d_t^mu\,\in\;L^\infty([0,T];H^1(\Omega))\,\cap\,W^{1,\infty}([0,T];L^2(\Omega))\,,
$$
so the trace of $\d_xv_m$ at any point $x\in\Omega$ is well-defined. To prove that it is in $L^2(0,T)$, i.e. that the right-hand side of \eqref{est:obs_log} is finite, we have to apply Proposition \ref{p:LL} or Proposition \ref{p:Z-LZ} (recall that $m=[\beta]+1$): we get
$$
\int_0^T\bigl|\d_t^m\d_xu(t,0)\bigr|\,dt\,\leq\,C\,(T+2T_\omega)\,E_{2m}(0)\,.
$$
Then, to conclude it is enough to notice that
$$
\d_tv_{2m\,|t=0}\,=\,\d_t^{2m+1}u_{|t=0}\,=\,D_\omega^{m}u_1\qquad\mbox{ and }\qquad
\d_xv_{2m\,|t=0}\,=\,\d_x\d_t^{2m}u_{|t=0}\,=\,\d_xD_\omega^{m}u_0
$$
both belong to $L^2(\Omega)$ thanks to our hypothesis.

\subsection{Sidewise energy estimates} \label{ss:sid}

Now, we apply the ``sidewise energy estimates'' technique. The fundamental observation is that the roles of the time and the space variables are interchangeable in equation \eqref{eq:we}. Hence, we can consider the ``new'' evolution problem
\begin{equation} \label{eq:sidewise}
        \d_x^2u\,-\,\omega(x)\,\d^2_tu\,=\,0\,,
\end{equation}
where this time we are looking at the evolution in $x$.

With this new point of view in mind, for any $k\in\N$ we define the \emph{sidewise $k$-energy} as
\begin{equation} \label{eq:sid_en}
F_k(x)\,:=\,\frac{1}{2}\,\int^{T-T_{\omega}x}_{T_{\omega}x}\left(\omega(x)\,\left|\d_t^{k+1}u(t,x)\right|^2\,+\,\left|\d_t^k\d_xu(t,x)\right|^2\right)dt\,,
\end{equation}
where $T_\omega$ was defined in \eqref{def:vel}. 
Let us immediately note that, thanks to \eqref{eq:hyp}, we have
$$
F_k(0)\,\sim\,\int^{T}_{0}\left|\d_t^k\d_xu(t,0)\right|^2\,dt\,.
$$

Our task is then finding suitable estimates for $F_k$. As the ``new'' equation \eqref{eq:sidewise} reads like \eqref{def:L}, we can use the results of subsection \ref{ss:cauchy}.
At this point, let us focus on the case of loss of derivatives, i.e. $\omega$ fulfilling an integral log-Lipschitz or log-Zygmund hypothesis.

Hence, we apply Propositions \ref{p:LL} and \ref{p:Z-LZ}-(ii) to equation \eqref{eq:sidewise}: from inequality \eqref{est:LL} with $s=1$ and the fact that $u(t,0)\equiv0$, we get
\begin{equation} \label{est:sid_1}
\left\|\d_tu(t,x)\right\|^2_{L^2(0,T)}\,+\,\left\|\d_xu(t,x)\right\|^2_{L^2(0,T)}\,\leq\,C\,\left\|\d_x u(t,0)\right\|^2_{H^\beta(0,T)}\,,
\end{equation}
Let us emphasize that the Sobolev norms which appear in this estimate are in time variable.

Now, we have to remark the following facts. 
\begin{enumerate}
\item $\|\d_xu(\,\cdot\,,0)\|_{H^\beta(0,T)}\leq\|\d_xu(\,\cdot\,,0)\|_{H^{[\beta]+1}(0,T)}$, where $[\beta]$ denotes the biggest integer smaller than or equal to $\beta$: we want to work with local operators, and so we need derivatives of integer order.
\item Let us focus on the term on the left-hand side of relation \eqref{est:sid_1}: by finite propagation speed and  the space reversibility of equation \eqref{eq:sidewise}, we have that, at any fixed point $x$, the previous integral is actually taken over the interval $[T_{\omega}x,T-T_{\omega}x]$:
$$
\left\|\d_tu(\,\cdot\,,x)\right\|^2_{L^2(0,T)}\,+\,\left\|\d_xu(\,\cdot\,,x)\right\|^2_{L^2(0,T)}\,=\,
\int_{T_{\omega}x}^{T-T_{\omega}x}\left(\left|\d_tu(t,x)\right|^2\,+\,\left|\d_xu(t,x)\right|^2\right)dt\,.
$$
\end{enumerate}

 Therefore, recalling also definition \eqref{eq:sid_en} and the hyperbolicity condition \eqref{eq:hyp} for $\omega$, from \eqref{est:sid_1} we infer the estimate
$$
F_0(x)\,\leq\,C\,\sum_{k=0}^{[\beta]+1}F_k(0)\,.
$$

Integrating the previous relation with respect to the space variable, we find
$$
\int_{T_{\omega}}^{T-T_{\omega}}\int_\Omega\left(\omega(x)\left|\d_tu(t,x)\right|^2+\left|\d_xu(t,x)\right|^2\right)dx\,dt\,\leq\,C\,\int^T_0\sum_{k=0}^{[\beta]+1}\left|\d_t^k\d_xu(t,0)\right|^2\,dt\,.
$$
At this point, we use Lemma \ref{l:en_E} for $E_0$: this leads us to
\begin{equation} \label{est:obs_partial}
\left(T\,-\,2\,T_{\omega}\right)\,E_0(0)\,\leq\,C\,\int^T_0\sum_{k=0}^{[\beta]+1}\left|\d_t^k\d_xu(t,0)\right|^2\,dt\,.
\end{equation}
Now, to obtain \eqref{est:obs_log}, it is enough to replace the non-homogeneous Sobolev norm with the homogeneous one, leaving just the highest order term $\d_t^m\d_xu(t,0)$
(see also Remark \ref{r:hom_sob} below).

Theorem \ref{th:LL-LZ} is now completely proved. 

\begin{rem} \label{r:hom_sob}
Let us stress again that it is
 possible to control the right-hand side of \eqref{est:obs_partial}
by the $L^2$-norm of just the $m$-th time derivative of $\d_xu_{|x=0}$.

Indeed, using the fact that the Fourier transform $\mc{F}:L^1\ra L^\infty$ is a continuous map
and that we are dealing with compactly supported (in time) functions, it is
 easy to see that
$$
\left\|\d_xu(\,\cdot\,,0)\right\|_{L^2(0,T)}\,\leq\,C_T\,\left\|\d^m_t\d_xu(\,\cdot\,,0)\right\|_{L^2(0,T)}\,.
$$

However, this relation can be recovered also using a sort of unique continuation property for our equation: as the proof is interesting in itself, let us sketch its main ideas. For simplicity, we will work on the whole $\R_t$, rather than on the interval $[0,T]$: this case can be recovered using finite propagation speed.

More precisely, arguing by contradiction, we are going to show that
\begin{equation} \label{est:low_freq}
\int_{B_t(0,1)}\bigl|\d_xu(t,0)\bigr|^2\,dt\,\leq\,C\,\int_{\R_t}\bigl|\d^m_t\d_xu(t,0)\bigr|^2\,dt
\end{equation}
for $u$ solution of system \eqref{eq:we}, with $m\geq1$.

In fact, suppose that the previous relation is not true; then (arguing in a sidewise manner) it is
 possible to find a sequence $\left(u^n\right)_n\,\subset\,\mc{C}(\Omega;H^1(0,T))$ of solutions to \eqref{eq:we}, such that
$$
\int_{B_t(0,1)}\bigl|\d_xu^n(t,0)\bigr|^2\,dt\,\equiv\,1\qquad\mbox{ and }\qquad\int_{\R_t}\bigl|\d^m_t\d_xu^n(t,0)\bigr|^2\,dt\,\longrightarrow\,0\,.
$$
In particular, from (sidewise) energy estimates we get that $\left(u^n\right)_n$ is a bounded sequence in $H^1$. Then, there exists a $u\in H^1(0,T)$ such that $u^n\rightharpoonup u$ in $L^\infty(\Omega;H^1(0,T))$ and which is still solution to \eqref{eq:we}. By Sobolev embeddings, one gathers strong convergence in $L^2(0,T)$, which entails
\begin{equation} \label{eq:cond_abs}
\int_{B_t(0,1)}\bigl|\d_xu(t,0)\bigr|^2\,dt\,=\,1\qquad\mbox{ and }\qquad\int_{\R_t}\bigl|\d^m_t\d_xu(t,0)\bigr|^2\,dt\,=\,0\,.
\end{equation}
In particular, these relations imply that $u\neq0$, but $\d^m_t\d_xu_{|x=0}=0$.

Define now $v_k\,:=\,\d^k_tu$ for any $0\leq k\leq m$. Applying the operator $\d^m_t$ to \eqref{eq:we} we find that $v_m$ solves the equation
$$
\d_x^2v_m\,-\,\omega(x)\,\d^2_tv_m\,=\,0\,,
$$
with $0$ initial data. By (sidewise) energy estimates, this implies that $v_m=\d^m_tu\equiv0$ over $\Omega\times[0,T]$, and then
$$
u(t,x)\,=\,k_0(x)\,+\,k_1(x)\,t\,+\,\ldots\,+k_{m-1}(x)\,t^{m-1}\,.
$$
From this relation and the boundary conditions $u(t,0)=u(t,1)\equiv0$ for all $t$, we infer that
$k_j(0)=k_j(1)=0$ for all $0\leq j\leq m-1$.

But $u$ has to fulfill equation \eqref{eq:we}: hence,
$$
\sum_{j=0}^{m-1}k_j''(x)\,t^{j}\,=\,\sum_{j=2}^{m-1}j\,(j-1)\,k_j(x)\,t^{j-2}\,.
$$
We now compare the terms of the same order (in $t$): in particular, we find
$$
k_{m-1}''(x)\,=\,0\qquad\Longrightarrow\qquad k_{m-1}(x)\equiv0 \quad\mbox{ in }\;\Omega
$$
(using also that $k_{m-1}(0)=k_{m-1}(1)=0$), and the same for $k_{m-2}$.

Then, it is
 easy to see that $k_j(x)\equiv0$ for all $0\leq j\leq m-1$, and so $u(t,x)\equiv0$, which is in contrast with the first condition in \eqref{eq:cond_abs}.

Inequality \eqref{est:low_freq} is then verified.
\end{rem}

Let us also spend a few words on the proof of Theorem \ref{th:Z}.
\begin{proof}[Proof of Theorem \ref{th:Z}]
To prove inequality \eqref{est:obs_Z}, one can just repeat the same arguments as above, i.e. sidewise energy estimates. But this time it is
 enough to consider the energies $E_0$ and $F_0$, because of the use of estimate \eqref{est:en_Z}, which involves no loss of regularity.
\end{proof}

\medbreak
We conclude the present section stating a result for moduli of continuity slightly better than the log-Lipschitz one. It is a direct consequence of Proposition \ref{p:L-LL} and sidewise energy estimates.
\begin{thm} \label{th:L+}
Let us consider the strictly hyperbolic problem \eqref{eq:we}-\eqref{eq:hyp}, with $T>2T_\omega$. Assume also that $\omega$ satisfies, for any $0<h<1$,
\begin{equation} \label{hyp:L+}
\sup_{x\in\R}\bigl|\omega(x+h)\,-\,\omega(x)\bigr|\,\leq\,K_0\,h\,\log\left(1+\frac{1}{h}\right)\,\mu(h)\,,
\end{equation}
where $\mu:[0,1]\ra[0,1]$ is a continuous strictly increasing  function such that $\mu(0)=0$.

Then, for any fixed $\delta>0$, there exists a constant $C>0$ (depending only on $T$, $\omega_*$, $\omega^*$, $\delta$ and $K_0$) such that the following inequality,
\begin{equation} \label{est:obs_small}
\left\|u_0\right\|^2_{H^1_0(\Omega)}\,+\,\left\|u_1\right\|^2_{L^2(\Omega)}\,\leq\,C\,\bigl\|\d_xu(\,\cdot\,,0)\bigr\|^2_{H^\delta(0,T)}\,,
\end{equation}
holds true for any initial data $u_0\in H^{3}(\Omega)\cap H^1_0(\Omega)$ and $u_1\in H^{2}(\Omega)$ such that
\begin{equation} \label{eq:D^m_omega}
D_\omega u_0\,\in\,H^1(\Omega)\,.
\end{equation}
\end{thm}

Note that $\delta$ can be taken arbitrarly small. However, in the following section we are going to prove that it cannot be taken equal to $0$, i.e. an estimate of the form \eqref{est:obs_Z} fails, in general, under assumption \eqref{hyp:L+}.

\section{Sharpness of the results} \label{s:sharp}

We want to discuss here the sharpness of our results.
In a first time, discuss the problem under Zygmund type assumptions.

Then, we will focus on the first variation of the coefficient: by construction of two
counterexamples, we will be able to prove that our results are optimal in this context, providing a full characterization of the observability estimates (with no, finite or inifite loss) depending on
the modulus of continuity of $\omega$.

In particular, this will imply a full classification also in the context of the second variation of the coefficients.

\medbreak
For convenience, throughout all this section we will focus on pointwise conditions.

\subsection{Remarks on Zygmund type conditions} \label{ss:second}

Let us consider Zygmund type conditions for the coefficient $\omega$. 

As recalled in subsections \ref{ss:LP} and \ref{ss:coeff}, by dyadic decomposition it is
 possible to characterize $Z$ as the Besov space $B^1_{\infty,\infty}$ and the H\"older class $\mc{C}^\alpha$, for any $0<\alpha<1$, as $B^\alpha_{\infty,\infty}$; $LZ$, instead, is equivalent to the logarithmic Besov space $B^{1-\log}_{\infty,\infty}$. Moreover, for any $0<\alpha_1<\alpha_2<1$ we have the following chain of strict embeddings:
\begin{equation} \label{eq:embedd}
B^1_{\infty,\infty}\,\hra\,B^{1-\log}_{\infty,\infty}\,\hra\,B^{\alpha_2}_{\infty,\infty}\,\hra\,B^{\alpha_1}_{\infty,\infty}\,.
\end{equation}

The previous relations \eqref{eq:embedd} suggest that our results, i.e. Theorems \ref{th:Z} and \ref{th:LL-LZ}, are optimal in this context.
In fact, we notice the following: given a modulus of continuity $\mu$, the following relation is trivial:
$$
\bigl|\omega(x+h)\,-\,\omega(x)\bigr|\,\leq\,C\,\mu(h)\qquad\Longrightarrow\qquad
\bigl|\omega(x+h)\,+\,\omega(x-h)\,-\,2\,\omega(x)\bigr|\,\leq\,2\,C\,\mu(h)\,.
$$
In particular, thanks to the counterexamples we will establish in the next subsection (see Theorems \ref{th:LL+} and \ref{th:L-LL}), this implies that our results are sharp, i.e. we have the classification:
\begin{itemize}
\item if $\omega$ is Zygmund continuous, then observability estimates hold true with no loss;
\item if $\omega$ verifies a condition (on second variation) beteween the Zygmund and the log-Zygmund ones, then observability estimates hold with a loss of derivatives, which cannot be avoided;
\item if the second variation of $\omega$ has a worse behaviour than the log-Zygmund one, then an infinite loss of derivatives in general occurs in observability estimates.
\end{itemize}

\subsection{About the first variation of the coefficients} \label{ss:first}

The present subsection is devoted to prove a complete characterization of the relation between observability estimates and modulus of continuity of $\omega$, in the sense that follows:
\begin{itemize}
\item if $\omega$ is Lipschitz continuous, there are observability estimates with no loss (we recall that this result is true also for $BV$ coefficients);
\item if the modulus of continuity of $\omega$ is strictly between the Lipschitz and the log-Lipschitz ones, a loss of derivatives has to occur, even if arbitrarly small (see Theorem \ref{th:L+});
\item if $\omega$ is even slightly worse than log-Lipschitz (for instance, H\"older), an infinite loss of derivatives occurs, and it is
 impossible to recover observability estimates.
\end{itemize}

Let us note that the $BV$ part was proved in \cite{FC-Z}, and then this result is sharp in the context of first variation of the coefficients.

The counterexample for H\"older continuous $\omega$, instead, was established in \cite{Castro-Z}. To prove the third point of the previous classification, we will improve it, verifying also that observability estimates fail even if we add the contribution of any high order time derivative of $\d_xu_{|x=0}$. In order to do this, we will resort also to some ideas of the corresponding counterexample for the Cauchy problem, which can be found in \cite{C-L}.

As a consequence, we get that Theorems \ref{th:LL-LZ} and \ref{th:L+} are sharp in the context of first variation.

Let us point out here some important issues.
\begin{rem} \label{r:counterex}
\begin{itemize}
\item[(i)] We will just prove the failure of observability estimates at $x=0$. However, repeating the same construction performed in \cite{Castro-Z}, it is
 possible to prove the analogous result also at $x=1$ and in the interior of $\Omega$.
\item[(ii)] In fact, in the third point of the classification, we are going to construct a coefficient which is smooth on the whole $\Omega$ but at $x=0$, where it presents a pathological behaviour. Also the initial data will be smooth out of the origin. So the additional compatibility conditions involving the operator $D_\omega$ will be automatically satisfied.
\end{itemize}
\end{rem}

Let us now state the main issue of this subsection. It asserts that an infinite loss of derivatives occurs whenever the modulus of continuity of $\omega$ is even slightly worse than the log-Lipschitz one. This result improves
the counterexample of \cite{Castro-Z} for H\"older continuous coefficients.
\begin{thm} \label{th:LL+}
Let $\psi:[1,+\infty[\,\ra[1,+\infty[\,$ be a strictly increasing and strictly concave function such that $\psi(1)=1$ and
\begin{equation} \label{eq:psi}
\lim_{\sigma\ra+\infty}\psi(\sigma)\,=\,+\infty\,.
\end{equation}

Then there exists $\omega$ satisfing \eqref{eq:hyp} and (for any $0<h<1/2$)
\begin{equation} \label{eq:omega-psi}
\sup_{x\in\R}\bigl|\omega(x+h)\,-\,\omega(x)\bigr|\,\leq\,K\,h\,\left|\log h\right|\,\psi\left(\left|\log h\right|\right)\,,
\end{equation}
and initial data $(u_0,u_1)\in H^1_0(\Omega)\times L^2(\Omega)$
for which observability estimates with finite loss cannot hold. More precisely, inequality \eqref{est:obs_log}
fails for any observability time $T$ and at any order $m\in\N$.
\end{thm}

In the intermediate instance between the Lipschitz and the log-Lipschitz ones, instead, we are able to prove that a loss of derivatives has to occur, in general.
\begin{thm} \label{th:L-LL}
Let $\lambda:\,]0,1]\ra[1,+\infty[\,$ be a continuous, strictly decreasing function such that $\lambda(1)=1$ and it fulfills the following conditions:
\begin{equation} \label{eq:lambda}
\lim_{\sigma\ra0^+}\lambda(\sigma)\,=\,+\infty\qquad\mbox{ but }\qquad\lim_{\sigma\ra0^+}\frac{\lambda(\sigma)}{\log\left(1+\frac{1}{\sigma}\right)}\,=\,0\,.
\end{equation}

Then there exists a sequence of density functions $\left(\omega_j\right)_{j\in\N}$, satisfying \eqref{eq:hyp} uniformly with respect to $j$ and also, for any $j\in\N$ and for any $0<h<1$,
\begin{equation} \label{eq:omega-lambda}
\sup_{x\in\R}\bigl|\omega_j(x+h)\,-\,\omega_j(x)\bigr|\,\leq\,K_j\;h\;\lambda(h)\,,
\end{equation}
and a sequence of initial data $\bigl(u^j_0,u^j_1\bigr)_{j\in\N}\subset H^1_0(\Omega)\times L^2(\Omega)$
such that the respective constants $C_j$ in \eqref{est:obs_Z} diverge to $+\infty$ for $j\ra+\infty$.
\end{thm}
As a corollary, we get that \eqref{est:obs_Z}, i.e. observability estimates without loss, fails in general for moduli of continuity which are worse than the Lipschitz one.

However, we will see that the constants $K_j$ in Theorem \ref{th:L-LL} diverge to $+\infty$ for $j\ra+\infty$: therefore the previous statement is somehow weaker than the one given in Theorem \ref{th:LL+}. As a matter of fact,
the sub-logarithmic divergence of $\lambda$ near $0$ is inadequate in order to construct a unique coefficient which presents strange pathologies in the concentration of the energy. So we
had to add some bad behaviour, which is responsible for the growth of the $K_j$'s. About this, note also that we couldn't argue by absurde and use the result of \cite{Cic-C} for the Cauchy problem: in order to have the null boundary condition at $x=1$ fulfilled, we would have to work in a time interval $[0,T]$ with $T<2T_\omega$, which is in contrast with our assumptions.

The proofs of Theorems \ref{th:LL+} and \ref{th:L-LL} will strictly follow the construction in \cite{Castro-Z} by Castro and Zuazua.
 The main difference is to find suitable oscillations, adapted to the new behaviours of the respective coefficients $\omega$, in order to reproduce
concentration phenomena which entail the loss of observability properties. At this step, let us notice that, under our assumptions, both $\psi$ in Theorem \ref{th:LL+} and $\lambda$ in Theorem \ref{th:L-LL} are invertible, and the inverse functions are still strictly monotone of the same type. This remark will be useful in the sequel.

\subsubsection{Proof of Theorem \ref{th:LL+}}

We will prove Theorem \ref{th:LL+} in various steps, following closely the construction of \cite{Castro-Z}.

In a first time, we will recall some preliminary results; then, we will construct simultaneously the coefficient $\omega$ and a sequence of ``quasi-eigenfuncions'' related to it, in the sense specified below. Finally, we will prove the lack of boundary observability.

\paragraph{\emph{Step 1: preparatory results.}}

For the sake of completeness, let us quote here two fundamental ODEs lemmas. For further discussions on them, we refer to \cite{C-S}, section 2, and \cite{Castro-Z}, section 2.
\begin{lemma} \label{l:ode_1}
There is a $\oline{\veps}>0$ such that, for all $\veps\in\,]0,\oline{\veps}[\,$, two even functions $\alpha_\veps$ and $w_\veps$ exist, which are $\mc{C}^\infty(\R)$, which fulfill
$$
\left\{\begin{array}{l}
w_\veps''\,+\,\alpha_\veps(x)\,w_\veps\,=\,0 \\[1ex]
w_\veps(0)\,=\,1\,,\quad w_\veps'(0)\,=\,0
\end{array}
\right. \leqno(O\!D\!E)
$$
and such that, for some positive constants $M$, $C$ and $\g$ independent of $\veps$, they verify:
\begin{itemize}
\item[(i)] $\alpha_\veps$ is $1$-periodic both on $x<0$ and $x>0$,
\item[(ii)] $\alpha_\veps\equiv4\pi^2$ in a neighborhood of $x=0$,
\item[(iii)] $\left|\alpha_\veps(x)-4\pi^2\right|\leq M\veps\quad$ and $\quad\left|\alpha_\veps'(x)\right|\leq M\veps$;
\vspace{.1cm}
\item[(A)] $w_\veps(x)=p_\veps(x)\, e^{-\veps|x|}$, where $p_\veps$ which is $1$-periodic on $x<0$ and $x>0$,
\item[(B)] $\left|w_\veps(x)\right|+\left|w'_\veps(x)\right|+\left|w''_\veps(x)\right|\leq C$,
\item[(C)] $\int_0^1w_\veps(x) dx\geq\g\veps$.
\vspace{0.5cm}
\end{itemize}
\end{lemma}

\begin{rem} \label{r:ode}
In particular, combining properties $(A)$ and $(ii)$ with $(O\!D\!E)$, we gather that
$$
w_\veps(x)\,=\,e^{-\veps|x|}\,,\qquad w_\veps'(x)\,=\,0\,,\qquad w_\veps''(x)\,=\,-4\pi^2\,e^{-\veps|x|}
$$
whenever $x=\pm n$, for all integers $n\geq0$.
\end{rem}

\begin{lemma} \label{l:ode_2}
Let $\phi$ be a solution of the equation (for $x\in\R$)
$$
\phi''\,+\,h^2\,\alpha(x)\,\phi\,=\,0\,,
$$
where $h\in\Z$ and $\alpha\in\mc{C}^1$ is a strictly positive function, and define the ``energies''
\begin{eqnarray*}
E_\phi(x) & := & 4\,\pi^2\,h^2\,\left|\phi(x)\right|^2\,+\,\left|\phi'(x)\right|^2 \\
\wtilde{E}_\phi(x) & := & 4\,h^2\,\alpha(x)\,\left|\phi(x)\right|^2\,+\,\left|\phi'(x)\right|^2\,.
\end{eqnarray*}

Then, for all $x_1$ and $x_2$ in $\R$, the following estimates holds true:
\begin{eqnarray*}
E_\phi(x_2) & \leq & E_\phi(x_1)\,\exp\left|h\,\int_{x_1}^{x_2}\left|4\,\pi^2\,-\,\alpha(x)\right|\,dx\right| \\
\wtilde{E}_\phi(x_2) & \leq & \wtilde{E}_\phi(x_1)\,\exp\left|\int_{x_1}^{x_2}\frac{\left|\alpha'(x)\right|}{\alpha(x)}\,dx\right|\,.
\end{eqnarray*}
\end{lemma}

We quote here also Proposition 1 of \cite{Castro-Z}, which will be very useful later on.
\begin{prop} \label{p:z}
Let $\omega\in L^\infty(\Omega)$, $0<\omega_*\leq\omega(x)\leq\omega^*$, and $T>0$ given. If $z$ satisfies
$$
\left\{\begin{array}{ll}
        \omega(x)\d_t^2z\,-\,\d^2_xz\,=\,0 & \quad\mbox{ in }\;\Omega\times\,]0,T[ \\[1ex]
        z(t,0)\,=\,f(t)\,,\quad z(t,1)\,=\,g(t) & \quad\mbox{ in }\;]0,T[ \\[1ex]
        z(0,x)\,=\,\d_tz(0,x)\,=\,0 & \quad\mbox{ in }\;\Omega\,,
       \end{array} 
\right.
$$
then there exists a constant $C(T)>0$ such that the following estimates hold true:
\begin{eqnarray*}
\int_0^T\int_\Omega\left(\omega(x)\,\bigl|\d_tz(t,x)\bigr|^2\,+\,\bigl|\d_xz(t,x)\bigr|^2\right)dx\,dt & \leq & C(T)\,\omega^*\left(\|f\|^2_{W^{2,\infty}(0,T)}\,+\,
\|g\|^2_{W^{2,\infty}(0,T)}\right) \\
\int_0^T\left(\bigl|\d_xz(t,0)\bigr|^2\,+\,\bigl|\d_xz(t,1)\bigr|^2\right)dx\,dt & \leq & C(T)\,\omega^*\left(\|f\|^2_{W^{3,\infty}(0,T)}\,+\,
\|g\|^2_{W^{3,\infty}(0,T)}\right)\,.
\end{eqnarray*}
\end{prop}

\paragraph{\emph{Step 2: the density $\omega$ and the quasi-eigenfunctions.}}
As in \cite{Castro-Z}, we will build simultaneously the density $\omega$ and a sequence of ``quasi-eigenfunctions'' (in the sense detailed below) which produce the concentration phenomenon we are looking for.

First of all, let us define a ``fractal'' partition of the interval $\Omega=[0,1]$. We set
\begin{equation} \label{eq:def_r}
r_j\,:=\,2^{-j}
\end{equation}
and we define a sequence of intervals $\left(I_j\right)_{j\geq2}$, centered in some $m_j$ and of lenght $r_j$:
$$
m_j\,:=\,\frac{r_j}{2}\,+\,\sum_{k=j+1}^{+\infty}r_k\qquad\mbox{ and }\qquad I_j\,:=\,\left]m_j\,-\,\frac{r_j}{2}\,,\,m_j\,+\,\frac{r_j}{2}\right]\,.
$$
In particular, we have $\,]0,1/2]=\bigcup_{j\geq2}I_j$.

Now, we define the positive sequences $\left(\veps_j\right)_j\searrow0^+\,$ and $\,\left(h_j\right)_j\nearrow+\infty\,$  by
\begin{equation} \label{eq:def_h-e}
h_j\,:=\,\exp\left(\psi^{-1}(2^{N j})\right)\qquad\mbox{ and }\qquad\veps_j\,h_j\,=\,\log h_j\;\psi\left(\log h_j\right)\,,
\end{equation}
for a suitable $N\in\N$ to be chosen later. We will see in a while that the particular definition of $\veps_j$ will imply the coefficient $\omega$ to have exactly the right modulus of continuity, i.e. to fulfill relation \eqref{eq:omega-psi}.
Let us note that these choices are coherent with the ones performed in \cite{Castro-Z}, which correspond to the case $\psi=\Id$. As a matter of fact, there the coefficient $\omega$ was H\"older continuous of any order because, actually, it was just slightly worse than log-Lipschitz (it had an extra logarithmic loss, represented by the $(\log)^2$ in the definition of $\veps_j$).

\begin{rem} \label{r:h-r}
Up to take the quantity\footnote{Recall that we denote by $[\sigma]$ the biggest integer less than or equal to $\sigma$.} $\left[\psi^{-1}(2^{N j})\right]+1$ in the definition of $h_j$ in \eqref{eq:def_h-e}, we can suppose that
$$
h_j\,r_j\,\equiv\,n_j\,\ra\,+\infty\,,\qquad\mbox{ with }\;\left(n_j\right)_j\,\subset\,\N\,.
$$
The claim $n_j\ra+\infty$ follows from the fact that, due to the hypothesis over $\psi$, its inverse $\psi^{-1}$ is strictly increasing and convex, so that
$\psi^{-1}(x)\geq k x$ for large $x$.
\end{rem}

We immediately fix $N$ big enough (with respect to the constant $M$ of Lemma \ref{l:ode_1}) such that
\begin{equation} \label{eq:cond_N}
\veps_j\,\leq\,\frac{1}{2\,M}\;,\qquad 5\,M\sum_{k=j+1}^{+\infty}\veps_k\,r_k\,\leq\,\veps_j\,r_j
\qquad\mbox{ and }\qquad 5\,M\sum_{k=1}^{j-1}\veps_k\,h_k\,r_k\,\leq\,\veps_j\,h_j\,r_j\,.
\end{equation}
Note that these inequalities are true for $N$ sufficiently large, due to the following considerations. Firstly, $\left(\veps_j\right)_j$ is decreasing to $0$ for any $N\geq2$, and
$\veps_1\ra0$ for $N\ra+\infty$. Moreover,
\begin{eqnarray*}
\veps_k\,r_k\,\left(\veps_j\,r_j\right)^{-1} & = & 2^{-(N+1)(k-j)}\;\frac{\psi^{-1}(2^{N k})}{\psi^{-1}(2^{N j})}\;\exp\left(-\,\psi^{-1}(2^{N k})\,+\,\psi^{-1}(2^{N j})\right) \\
\veps_k\,h_k\,r_k\,\left(\veps_j\,h_j\,r_j\right)^{-1} & = & 2^{(N-1)(k-j)}\;\frac{\psi^{-1}(2^{N k})}{\psi^{-1}(2^{N j})}\,;
\end{eqnarray*}
therefore, it is
 easy to see that $\veps_k\,r_k\left(\veps_j\,r_j\right)^{-1}\leq\delta_N^{(k-j)}$, with $\delta_N\ra0$ for $N\ra+\infty$, and that
$\veps_k\,h_k\,r_k\left(\veps_j\,h_j\,r_j\right)^{-1}\leq\Gamma_N^{(k-j)}$, with $\Gamma_N\ra+\infty$ for $N\ra+\infty$.

We can now define the coefficient $\omega$: taking the $\alpha_{\veps_j}$'s introduced in Lemma \ref{l:ode_1}, we set
\begin{equation} \label{eq:def_omega}
\omega(x)\,:=\,\left\{\begin{array}{ll}
\alpha_{\veps_j}\left(h_j\,(x-m_j)\right) & \quad\mbox{if }\;x\,\in\,I_j \\[1ex]
4\,\pi^2 & \quad\mbox{if }\;x\,\in\,[0,1]\setminus\bigcup_{j\geq2}I_j\,.
\end{array}
\right.
\end{equation}

Note that $\omega\in\mc{C}^\infty$ in the interval $\,]0,1]$, because $\alpha_\veps\in\mc{C}^\infty(\R)$ and $\omega\equiv4\pi^2$ in a neighborhood of the extremes
of $I_j$, due to properties $(i)$ and $(ii)$ of Lemma \ref{l:ode_1} (see also Remark \ref{r:h-r}).

We claim that $\omega$, defined in this way, satisfies inequality \eqref{eq:omega-psi} over the whole $[0,1]$. Indeed, thanks to properties $(i)$ and $(iii)$ of Lemma \ref{l:ode_1}, we have (recall that $|x-y|<1/2$)
\begin{eqnarray*}
\left|\alpha_{\veps_j}\right|_{\psi,\R} & := & \sup_{x,y\in\R}\frac{\left|\alpha_{\veps_j}(x)-\alpha_{\veps_j}(y)\right|}{|x-y|\,\bigl|\log|x-y|\bigr|\,\psi\left(\bigl|\log|x-y|\bigr|\right)} \\
& \leq &  \sup_{x,y\in[0,1]}\frac{\left|\alpha_{\veps_j}(x)-\alpha_{\veps_j}(y)\right|}{|x-y|\,\bigl|\log|x-y|\bigr|\,\psi\left(\bigl|\log|x-y|\bigr|\right)}\;\leq\;\,C\,M\,\veps_j\,.
\end{eqnarray*}
Therefore, for all $j\geq2$, we get
\begin{eqnarray*}
\left|\omega\right|_{\psi,I_j} & := & \sup_{x,y\in I_j}\frac{\left|\omega(x)-\omega(y)\right|}{|x-y|\,\bigl|\log|x-y|\bigr|\,\psi\left(\bigl|\log|x-y|\bigr|\right)} \\
& \leq &  \sup_{x,y\in\R}\frac{\left|\alpha_{\veps_j}(h_j(x-m_j))-\alpha_{\veps_j}(h_j(y-m_j))\right|}{|x-y|\,\bigl|\log|x-y|\bigr|\,\psi\left(\bigl|\log|x-y|\bigr|\right)} \\
& \leq &  h_j\;\sup_{h_j|x-y|<1/2}\frac{\left|\alpha_{\veps_j}(h_j(x-m_j))-\alpha_{\veps_j}(h_j(y-m_j))\right|}{h_j\,|x-y|\,\bigl|\log|x-y|\bigr|\,\psi\left(\bigl|\log|x-y|\bigr|\right)} \\
& \leq & \frac{C\,M\,\veps_j\,h_j}{\log h_j\;\psi\left(\log h_j\right)}\;\leq\;C\,M\,.
\end{eqnarray*}
Now we set $I_z$ to be the interval of the family $\left(I_j\right)_j$ such that $z\in I_z$, and we denote by $\ell_z$ and $r_z$ its left and right extremes respectively. Then
\begin{eqnarray}
\sup_{x,y\in[0,1]}\frac{\left|\omega(x)-\omega(y)\right|}{|x-y|\,\bigl|\log|x-y|\bigr|\,\psi\left(\bigl|\log|x-y|\bigr|\right)} & \leq &
\sup_{0\leq x\leq y\leq1}\frac{\left|\omega(x)-\omega(r_x)\right|+\left|\omega(\ell_y)-\omega(y)\right|}{|x-y|\,\bigl|\log|x-y|\bigr|\,\psi\!\left(\bigl|\log|x-y|\bigr|\right)} \label{est:mod_cont} \\
& \leq & 2\,\sup_{j\geq2}\left|\omega\right|_{\psi,I_j}\,, \nonumber
\end{eqnarray}
where we have used also the fact that $\omega(r_x)=\omega(\ell_y)$ (see also Remark \ref{r:h-r}).
So, we gather that $\omega$ fulfills \eqref{eq:omega-psi} with $K=2CM$ over the whole $[0,1]$.

Finally, thanks to issue $(ii)$ in Lemma \ref{l:ode_1} and the first relation in \eqref{eq:cond_N}, we have also
$$
2\,\pi^2\,\leq\,\omega(x)\,\leq\,8\,\pi^2\qquad\qquad\forall\;x\,\in\,[0,1].
$$

The following step consists then in defining the sequence of quasi-eigenfunctions $\left(\phi_j\right)_{j\geq2}$ as the solutions of
\begin{equation} \label{eq:def_phi}
\left\{\begin{array}{ll}
\phi_j''\,+\,h_j^2\,\omega(x)\,\phi_j\,=\,0 & \qquad\mbox{ for }\;x\,\in\,[0,1] \\[1ex]
\phi_j(m_j)\,=\,1\,,\quad\phi'_j(m_j)\,=\,0\,. & 
\end{array}\right.
\end{equation}
As already pointed out in \cite{Castro-Z}, these functions do not satisfy the null boundary conditions at $x=0$ and $x=1$. However, as we will see in a while, they are concentrated in the interior of each interval
$I_j$, so that their values at the extremes of $[0,1]$ are exponentially small. This is the reason why we refer to them as ``quasi-eigenfunctions''.

Let us observe that, due to the definition of $\omega$ in \eqref{eq:def_omega}, $\phi_j$ satisfies
\begin{equation} \label{eq:phi-I_j}
\left\{\begin{array}{l}
\phi_j''\,+\,h_j^2\,\alpha_{\veps_j}\bigl(h_j(x-m_j)\bigr)\,\phi_j\,=\,0 \\[1ex]
\phi_j(m_j)\,=\,1\,,\quad\phi'_j(m_j)\,=\,0 
\end{array}\right.
\end{equation}
for $x\in I_j$, which implies that
\begin{equation} \label{eq:phi-w}
\phi_j(x)\,\equiv\,w_{\veps_j}\bigl(h_j(x-m_j)\bigr)
\end{equation}

in the interval $I_j$, where $w_{\veps_j}$ are the functions introduced in Lemma \ref{l:ode_1}, associated to $\alpha_{\veps_j}$. Therefore, combining Remarks \ref{r:ode} and \ref{r:h-r},
after easy computations we infer that
\begin{eqnarray}
\int_{I_j}\left|\phi_j(x)\right|^2\,dx & \geq & \frac{C}{h_j^3} \label{est:en_interior} \\
\bigl|\phi_j(m_j\,\pm\,r_j/2)\bigr|^2\,+\,\bigl|\phi'_j(m_j\,\pm\,r_j/2)\bigr|^2 & = & \exp\bigl(-\,\veps_j\,h_j\,r_j\bigr)\,. \label{est:en_extreme}
\end{eqnarray}
Due to the definitions \eqref{eq:def_r} and \eqref{eq:def_h-e} of our sequences, we have
$$
\frac{\veps_j\,h_j\,r_j\,}{\log h_j}\,=\,2^{-j}\;\psi\left(\log h_j\right)\,=\,2^{(N-1)j}\,\longrightarrow\,+\infty\qquad\mbox{ for }\;j\,\ra\,+\infty\,.
$$
From this property, we deduce that, roughly speaking, the energy of the quasi-eigenfunctions $\phi_j$ is concentrated inside the subintervals $I_j$. Formally, for any fixed
$p>0$,
\begin{equation} \label{est:en_conc}
\exp\bigl(-\,\veps_j\,h_j\,r_j\bigr)\,\leq\,C_p\,h_j^{-p}\,.
\end{equation}

We will show now that the energy remains exponentially small also at $x=0$ and $x=1$. As a matter of fact, we fix a $x\leq m_j-r_j/2$ and we estimate
$$
E_{\phi_j}(x)\,:=\,4\,\pi^2\,h_j^2\,\left|\phi_j(x)\right|^2\,+\,\left|\phi_j'(x)\right|^2\,.
$$
By Lemma \ref{l:ode_2}, and recalling the definition of $\omega$ in \eqref{eq:def_omega} and property $(iii)$ of Lemma \ref{l:ode_1}, we infer
\begin{eqnarray*}
E_{\phi_j}(x) & \leq & E_{\phi_j}(m_j-r_j/2)\;\exp\left(h_j\int_x^{m_j-r_j/2}\bigl|4\,\pi^2\,-\,\omega(y)\bigr|\,dy\right) \\
&\leq &  E_{\phi_j}(m_j-r_j/2)\;\exp\left(M\,h_j\,\sum_{k=j+1}^{+\infty}\veps_k\,r_k\right)\,.
\end{eqnarray*} 
Taking into account \eqref{est:en_extreme}, in order to control the first term on the right hand side, and the second condition in \eqref{eq:cond_N}, in order to control
the series in the exponential term, we find
\begin{eqnarray*}
E_{\phi_j}(x) & \leq & 4\,\pi^2\,h_j^2\,\exp\bigl(-\,\veps_j\,h_j\,r_j\bigr)\;\exp\left(\frac{1}{5}\,h_j\,\veps_j\,r_j\right) \\
&\leq &  4\,\pi^2\,h_j^2\;\exp\left(-\,\frac{4}{5}\,\veps_j\,h_j\,r_j\right)\,.
\end{eqnarray*} 
Therefore, we have found that for any $x\leq m_j-r_j/2$, for all $p>0$,
\begin{equation} \label{est:en_0}
\bigl|\phi_j(x)\bigr|^2\,+\,\bigl|\phi_j'(x)\bigr|^2\,\leq\,C_p\,h_j^{-p}\,;
\end{equation}
in particular, this relation holds true also at $x=0$.

Let us perform the analysis at the point $x=1$. As $\omega\equiv4\pi^2$ in $\,]1/2,1]$, if we define
$$
\wtilde{E}_{\phi_j}(x)\,:=\,h_j^2\,\omega(x)\,\left|\phi_j(x)\right|^2\,+\,\left|\phi_j'(x)\right|^2\,,
$$
Lemma \ref{l:ode_2} gives us $\wtilde{E}_{\phi_j}(1)\,\leq\,\wtilde{E}_{\phi_j}(1/2)$. Now, for any $m_j+r_j/2\leq x\leq1/2$, we use the same lemma: keeping in mind also definition \eqref{eq:def_omega}, the positivity of $\omega$ and property $(iii)$ in Lemma \ref{l:ode_1}, we discover that
\begin{eqnarray*}
\wtilde{E}_{\phi_j}(x) & \leq & \wtilde{E}_{\phi_j}(m_j+r_j/2)\;\exp\left(\int_{m_j+r_j/2}^x\frac{\left|\omega'(y)\right|}{\omega(y)}\;dy\right) \\
& \leq & \wtilde{E}_{\phi_j}(m_j+r_j/2)\;\exp\left(M\sum_{k=1}^{j-1}\veps_k\,h_k\,r_k\right) \\
& \leq & h_j^2\,\exp\left(-\,\frac{4}{5}\,\veps_j\,h_j\,r_j\right)\,,
\end{eqnarray*}
where in the last step we have used also the third condition in \eqref{eq:cond_N} and estimate \eqref{est:en_extreme}. Then, for any $m_j+r_j/2\leq x\leq1/2$,
we get
\begin{equation} \label{est:en_1}
\bigl|\phi_j(x)\bigr|^2\,+\,\bigl|\phi_j'(x)\bigr|^2\,\leq\,C_p\,h_j^{-p}\,;
\end{equation}
for all $p>0$; in particular, \eqref{est:en_1} holds true also at $x=1/2$, and therefore at $x=1$.

\paragraph{\emph{Step 3: lack of the boundary observability.}}
We now prove that observability estimate \eqref{est:obs_log} fails for the density $\omega$ constructed in the previous step, for any $T>0$ and any $m\in\N$.

Note that, thanks to the definition of $\left(\phi_j\right)_j$, for any $j\in\N$ the function
$$
v_j(t,x)\,:=\,\phi_j(x)\,e^{i h_j t}
$$
is a solution of the first equation in $\eqref{eq:we}$, but it does not necessarily fulfill the null boundary condition at $x=0$ and $x=1$, because the $\phi_j$'s don't.

So, let us define the sequence $\left(z_j\right)_j$ in the following way: for any $j$, $z_j$ is the unique solution to the system
\begin{equation} \label{eq:def_z}
\left\{\begin{array}{ll}
        \omega(x)\d_t^2z_j\,-\,\d^2_xz_j\,=\,0 & \quad\mbox{ in }\;\Omega\times\,]0,T[ \\[1ex]
        z_j(t,0)\,=\,-\,v_j(t,0)\,=\,-\,e^{i h_j t}\,\phi_j(0) & \quad\mbox{ in }\;]0,T[ \\[1ex]
z_j(t,1)\,=\,-\,v_j(t,1)\,=\,-\,e^{i h_j t}\,\phi_j(1) & \quad\mbox{ in }\;]0,T[ \\[1ex]
        z_j(0,x)\,=\,\d_tz_j(0,x)\,=\,0 & \quad\mbox{ in }\;\Omega\,.
       \end{array} 
\right.
\end{equation}
Then, defined $u_j$ (for all $j\geq1$) as
$$
u_j\,:=\,v_j\,+\,z_j\,,
$$
 the sequence $\left(u_j\right)_j$ is a family of true solutions to system \eqref{eq:we}. Our aim is now proving that
\begin{equation} \label{lim:LL+}
\lim_{j\ra+\infty}\;\frac{\int_\Omega\left(\omega(x)\,\bigl|\d_tu_j(0,x)\bigr|^2\,+\,\bigl|\d_xu_j(0,x)\bigr|^2\right)dx}{\int_0^T\bigl|\d_t^m\d_xu_j(t,0)\bigr|^2\,dt}\;=\;
+\infty
\end{equation}
for any $T>0$ and any $m\in\N$.

Let us start by considering the numerator. By definitions and the lower bound for $\omega$, we have
\begin{eqnarray*}
\int_\Omega\left(\omega(x)\,\bigl|\d_tu_j(0,x)\bigr|^2\,+\,\bigl|\d_xu_j(0,x)\bigr|^2\right)dx & = &
\left(\omega(x)\,\bigl|\d_tv_j(0,x)\bigr|^2\,+\,\bigl|\d_xv_j(0,x)\bigr|^2\right)dx \\
& = & \int_\Omega\left(h_j^2\,\omega(x)\,\bigl|\phi_j(x)\bigr|^2\,+\,\bigl|\phi_j'(x)\bigr|^2\right)dx \\
& \geq & 2\,\pi^2\,h_j^2\int_{\Omega}\bigl|\phi_j(x)\bigr|^2\,dx\,.  
\end{eqnarray*}
Taking the integral just on the interval $I_j$ rather than on the whole $\Omega=[0,1]$, and recalling estimate \eqref{est:en_interior}, it follows that
\begin{equation} \label{est:LL+_num}
\int_\Omega\left(\omega(x)\,\bigl|\d_tu_j(0,x)\bigr|^2\,+\,\bigl|\d_xu_j(0,x)\bigr|^2\right)dx\,\geq\,C\,h_j^{-1}\,.
\end{equation}

We now focus on the denominator in \eqref{lim:LL+}:
\begin{eqnarray*}
\int_0^T\bigl|\d_t^m\d_xu_j(t,0)\bigr|^2\,dt & \leq & \int_0^T\left(\bigl|\d_t^m\d_xv_j(t,0)\bigr|^2\,+\,\bigl|\d_t^m\d_xz_j(t,0)\bigr|^2\right)dt \\
& \leq & T\,\bigl|\phi_j'(0)\bigr|^2\,\left(h_j\right)^m\,+\,\int_0^T\bigl|\d_t^m\d_xz_j(t,0)\bigr|^2\,dt\,.
\end{eqnarray*}
Keeping in mind \eqref{eq:def_z}, it is
 possible to see that the derivative $\d_t^mz_j$ fulfills the system
$$
\left\{\begin{array}{ll}
        \omega(x)\d_t^2\left(\d_t^mz_j\right)\,-\,\d^2_x\left(\d_t^mz_j\right)\,=\,0 & \quad\mbox{ in }\;\Omega\times\,]0,T[ \\[1ex]
        \d_t^mz_j(t,0)\,=\,-\,\left(i\,h_j\right)^m\,e^{i h_j t}\,\phi_j(0) & \quad\mbox{ in }\;]0,T[ \\[1ex]
\d_t^mz_j(t,1)\,=\,-\,\left(i\,h_j\right)^m\,e^{i h_j t}\,\phi_j(1) & \quad\mbox{ in }\;]0,T[ \\[1ex]
        \d_t^mz_j(0,x)\,=\,\d_t^{m+1}z_j(0,x)\,=\,0 & \quad\mbox{ in }\;\Omega\,.
       \end{array} 
\right.
$$
Then, applying Proposition \ref{p:z} entails
$$
\int^T_0\bigl|\d_t^m\d_xz_j(t,0)\bigr|^2\,dt\,\leq\,C(T)\;\omega^*\,\left(h_j\right)^{2(m+3)}\left(\bigl|\phi_j(0)\bigr|^2\,+\,\bigl|\phi_j(1)\bigr|^2\right)\,.
$$
Inserting this relation into the previous one, we find
\begin{equation} \label{est:LL+_den}
\int_0^T\bigl|\d_t^m\d_xu_j(t,0)\bigr|^2\,dt\,\leq\,C(T,\omega^*)\,\left(h_j\right)^{2(m+3)}\left(\bigl|\phi_j(0)\bigr|^2\,+\,\bigl|\phi_j(1)\bigr|^2\,+\,\bigl|\phi_j'(0)\bigr|^2\right)\,.
\end{equation}

Now, recall that inequality \eqref{est:en_0} is true both at $x=0$ and $x=1$. So, if we put estimates \eqref{est:LL+_num} and \eqref{est:LL+_den} together, for all
$T>0$, all $m\in\N$ and all $p>0$ we find, for a suitable $C_p>0$,
$$
\frac{\int_\Omega\left(\omega(x)\,\bigl|\d_tu_j(0,x)\bigr|^2\,+\,\bigl|\d_xu_j(0,x)\bigr|^2\right)dx}{\int_0^T\bigl|\d_t^m\d_xu_j(t,0)\bigr|^2\,dt}\;\geq\;
\frac{C\,h_j^{-1}}{C(T,\omega^*)\,\left(h_j\right)^{2(m+3)}\,C_p\,h_j^{-p}}\;,
$$
which diverges to $+\infty$ if we take $p>2m+7$.

Relation \eqref{lim:LL+} is then satisfied, and Theorem \ref{th:LL+} is completely proved.

\subsubsection{Proof of Theorem \ref{th:L-LL}}

Let us spend here a few words on the proof of Theorem \ref{th:L-LL}.

It can be carried out following the same arguments as before, even if we have to pay attention to the choice
of the sequences $\left(\veps_j\right)_j$ and $\left(h_j\right)_j$.

So, we define $\left(r_j\right)_j$ and the subintervals $\left(I_j\right)_j$ as before, but we replace \eqref{eq:def_h-e} by
$$
h_j\,:=\,\frac{1}{\lambda^{-1}\left(2^{N j}\right)}\qquad\mbox{ and }\qquad\veps_j\,h_j\,=\,\lambda(1/h_j)\,\log h_j\,,
$$
where $N$ is a suitable integer, to be fixed later.

Notice that, by hypothesis, $\log\sigma\geq\lambda(1/\sigma)$ for $\sigma\geq\sigma_0$. Taking $\sigma=e^x$ and recalling that $\lambda^{-1}$ is decreasing,
we infer
$$
e^x\,\leq\,\frac{1}{\lambda^{-1}(x)}\qquad\Longrightarrow\qquad h_j\,\geq\,\exp\left(2^{N j}\right)\,.
$$
From this we get, in particular, that $h_j\;r_j\,\ra\,+\infty$, as in Remark \ref{r:h-r}. Moreover, up to take the entire parts,
we can suppose $\left(h_j\,r_j\right)_j\,\subset\,\N$.

It is
 also easy to see (exactly as before) that we can fix $N\in\N$ large enough, such that
$$
\veps_j\,\leq\,\frac{1}{2\,M}\;,\qquad 5\,M\sum_{k=j+1}^{+\infty}\veps_k\,r_k\,\leq\,\veps_j\,r_j
\qquad\mbox{ and }\qquad 5\,M\sum_{k=1}^{j-1}\veps_k\,h_k\,r_k\,\leq\,\veps_j\,h_j\,r_j\,.
$$

Now, we can define the sequence $\left(\omega_j\right)_{j\in\N}$ in the following way: for any $j\in\N$ we set
\begin{equation} \label{eq:def_omega_lambda}
\omega_j(x)\,:=\,\left\{\begin{array}{ll}
\alpha_{\veps_j}\bigl(h_j\,(x-m_j)\bigr) & \quad\mbox{if }\;x\,\in\,I_j \\[1ex]
4\,\pi^2 & \quad\mbox{if }\;x\,\in\,[0,1]\setminus I_j\,.
\end{array}
\right.
\end{equation}
The properties of $\omega_j$ are of easy verification, and they can be proved arguing as above. In particular, we have that $\omega_j$ fulfills \eqref{eq:omega-lambda} with
$K_j=C M \log h_j$, for some suitable constant $C$ independent of $j\in\N$.

Then we define the sequence of quasi-eigenfunctions $\left(\phi_j\right)_{j\geq2}$ as the solutions of
\begin{equation} \label{eq:def_phi-lambda}
\left\{\begin{array}{ll}
\phi_j''\,+\,h_j^2\,\omega_j(x)\,\phi_j\,=\,0 & \qquad\mbox{ for }\;x\,\in\,[0,1] \\[1ex]
\phi_j(m_j)\,=\,1\,,\quad\phi'_j(m_j)\,=\,0\,. & 
\end{array}\right.
\end{equation}
Due to the definition of $\omega_j$ in \eqref{eq:def_omega_lambda}, for $x\in I_j$, $\phi_j$ satisfies
$$   
\left\{\begin{array}{l}
\phi_j''\,+\,h_j^2\,\alpha_{\veps_j}\bigl(h_j\,(x-m_j)\bigr)\,\phi_j\,=\,0 \\[1ex]
\phi_j(m_j)\,=\,1\,,\quad\phi'_j(m_j)\,=\,0\,,
\end{array}\right.
$$   
which implies, exactly as before, that
$$   
\phi_j(x)\,\equiv\,w_{\veps_j}\bigl(h_j\,(x-m_j)\bigr)
$$   
in the interval $I_j$, where $w_{\veps_j}$ are the functions introduced in Lemma \ref{l:ode_1}, associated to $\alpha_{\veps_j}$. Therefore, also in this case we infer that
\begin{eqnarray}
\int_{I_j}\left|\phi_j(x)\right|^2\,dx & \geq & \frac{C}{h_j^3} \label{est:en_interior-lambda} \\
\bigl|\phi_j(m_j\,\pm\,r_j/2)\bigr|^2\,+\,\bigl|\phi'_j(m_j\,\pm\,r_j/2)\bigr|^2 & = & \exp\bigl(-\,\veps_j\,h_j\,r_j\bigr)\,. \label{est:en_extreme-lambda}
\end{eqnarray}
Due to the definitions of our sequences, we have moreover that
\begin{equation} \label{eq:divergence_lambda}
\frac{\veps_j\,h_j\,r_j\,}{\log h_j}\,=\,\lambda\bigl(1/h_j\bigr)\,r_j\,=\,2^{(N-1)j}\,\longrightarrow\,+\infty\qquad\mbox{ for }\;j\,\ra\,+\infty\,.
\end{equation}
Once again, this property tells us that the energy of the quasi-eigenfunctions $\phi_j$ is concentrated inside the subintervals $I_j$.

Moreover, as done before, it is
 easy to see that the $j$-th energy is small at $x=0$ and $x=1$:
$$
\bigl|\phi_j(x)\bigr|^2\,+\,\bigl|\phi'_j(x)\bigr|^2\,\leq\,C\,h_j^2\,\exp\biggl(-\,\frac{4}{5}\,\veps_j\,h_j\,r_j\biggr)\,.
$$

Now, the proof of the lack of the boundary observability follows as before: we define
$$
v_j(t,x)\,:=\,\phi_j(x)\,e^{i h_j  t}\qquad\mbox{ and }\qquad u_j\,:=\,v_j\,+\,z_j\,,
$$
where $z_j$ is still defined by \eqref{eq:def_z}.
Actually we don't have to consider higher horder derivatives, and we have just to show that, for all $T>0$,
$$
\lim_{j\ra+\infty}\;\frac{\int_\Omega\left(\omega(x)\,\bigl|\d_tu_j(0,x)\bigr|^2\,+\,\bigl|\d_xu_j(0,x)\bigr|^2\right)dx}{\int_0^T\bigl|\d_xu_j(t,0)\bigr|^2\,dt}\;=\;
+\infty.
$$

Aguing exactly as before and keeping in mind the estimates for the energy inside the intervals $I_j$'s and at the extremities $x=0$ and $x=1$, we are led to the inequality
$$
\frac{\int_\Omega\left(\omega(x)\,\bigl|\d_tu_j(0,x)\bigr|^2\,+\,\bigl|\d_xu_j(0,x)\bigr|^2\right)dx}{\int_0^T\bigl|\d_xu_j(t,0)\bigr|^2\,dt}\,\geq\,\frac{C\,h_j^{-1}\,\exp\left(C'\,\veps_j\,h_j\,r_j\right)}{C''\,h_j^6}\,,
$$
for some suitable positive constants $C$, $C'$ and $C''$. At this point, we use \eqref{eq:divergence_lambda}, which completes the proof of Theorem \ref{th:L-LL}.

\section{Controllability results} \label{s:control}

As discussed in the introduction, the observability results stated in this paper have direct counterparts at the level of controllability: the present section is devoted to them.

Consider the controlled wave equation:
\begin{equation} \label{eq:control}
\left\{\begin{array}{ll}
        \omega(x)\d_t^2y\,-\,\d^2_xy\,=\,0 & \quad\mbox{ in }\;\Omega\times\,]0,T[ \\[1ex]
        y(t,0)\,= f(t), \quad \,y(t,1)\,=\,0 & \quad\mbox{ in }\;]0,T[ \\[1ex]
        y(0,x)\,=\,y_0(x)\,,\quad \d_ty(0,x)\,=\,y_1(x) & \quad\mbox{ in }\;\Omega\,.
       \end{array}
\right.
\end{equation}
We are interested in the problem of \emph{null controllability}: namely, whether for a given initial datum $(y_0, y_1)$ there exists a control $f=f(t)$ such that the solution of \eqref{eq:control} satisfies
\begin{equation} \label{eq:control_final}
  y(T,x)\,=\,\d_ty(T,x)\,=\,0 \qquad\mbox{ in }\;\Omega\,.
\end{equation}

Let us remark that, in general, the product by a coefficient $\omega\in\mc{A}$ is not continuous from $H^s$ into itself. For instance, the product by a log-Lipschitz function maps continuously $H^s\ra H^s$ if and only if $|s|<1$ (see \cite{C-L}). For this reason, as already done in e.g. \cite{Castro-Z} and \cite{FC-Z}, we will assume a priori that $\omega y_1$ belongs to $H^{-1}$.

First of all, we analyze the consequences of the sharp observability result in Theorem \ref{th:Z}. In this case, it is well known that the observability inequality \eqref{est:obs_Z} is equivalent to the fact that for all $(y_0, \omega y_1) \in L^2(\Omega) \times H^{-1}(\Omega)$ there exists a control $f \in L^2(0, T)$ such that the solution of \eqref{eq:control} satisfies \eqref{eq:control_final} (see also \cite{Z-2006}).
\begin{thm} \label{th:control_Z}
Consider system \eqref{eq:control}, where the coefficient $\omega$ satisfies the hyperbolicity condition \eqref{eq:hyp} and the integral Zygmund assumption \eqref{est:Z_1}. Let the time $T>2T_\omega$, where $T_\omega$ is defined by \eqref{def:vel}.

Then, for any $(y_0, \omega y_1) \in L^2(\Omega) \times H^{-1}(\Omega)$, there exists a control
$f\in L^2(0,T)$ such that the solution $y$ to system \eqref{eq:control} satisfies conditions
\eqref{eq:control_final}. \\
Moreover, there exists a constant $C$, just depending on $T$ and on $\omega_*$, $\omega^*$ and $|\omega|_{Z}$, such that one can choose $f$ fulfilling
$$
\|f\|_{L^2(0,T)}\,\leq\,C\,\biggl(\left\|y_0\right\|_{L^2(\Omega)}\,+\,\left\|y_1\right\|_{H^{-1}(\Omega)}\biggr)\,.
$$
\end{thm}

Let us now turn our attention to the estimate in Theorem \ref{th:LL-LZ}, where the observability is proved with a loss of a finite number of derivates. At the control level, this is equivalent to the fact that all initial data $(y_0,\omega  y_1) \in L^2(\Omega) \times H^{-1}(\Omega)$ are controllable with controls in a larger space  $f \in H^{-m}(0, T)$.

\begin{thm} \label{th:control_loss}
Assume, this time, that the coefficient $\omega$ belongs to the integral classes $LL$ or $LZ$, defined respectively by conditions \eqref{est:LL_1} and \eqref{est:LZ_1}.  Let the time $T>2T_\omega$.

Then, for any $(y_0, \omega y_1) \in L^2(\Omega) \times H^{-1}(\Omega)$, there exist a $m\in\N$ and a control
$f\in H^{-m}(0,T)$ such that the solution $y$ to system \eqref{eq:control} satisfies conditions
\eqref{eq:control_final}. \\
Moreover, there exists a constant $C$, just depending on $T$ and on $\omega_*$, $\omega^*$ and $|\omega|_{LL}$ or $|\omega|_{LZ}$ respectively, such that one can choose $f$ fulfilling
$$
\|f\|_{H^{-m}(0,T)}\,\leq\,C\,\biggl(\left\|y_0\right\|_{L^2(\Omega)}\,+\,\left\|y_1\right\|_{H^{-1}(\Omega)}\biggr)\,.
$$
\end{thm}

\begin{rem} \label{r:control_loss}
In view of Theorem \ref{th:L-LL}, we must have $m>0$.
\end{rem}

Note that, as showed in \cite{D-Z}, similar phenomena occur, for instance, in the control of the wave equation in planar networks, where, depending on the diophantine approximation properties of the mutual lengths of the strings entering in the network, a finite number of derivatives can be lost both at the observation and control level.

Similarly, as a consequence of the negative results in section \ref{s:sharp}, one can deduce that when the coefficients of the equation are too rough for observability inequalities to hold, the same occurs at the controllability level. In other words, the controllability results above fail whenever the observability ones do it.
\begin{thm} \label{th:no_control}
Let $\psi$ be as in the hypothesis of Theorem \ref{th:LL+}.

There exists a density function $\omega$ which verifies conditions \eqref{eq:hyp} and \eqref{eq:omega-psi}, and such that the following fact holds true: for any control time $T>0$, there exist initial data $(y_0, \omega y_1) \in L^2(\Omega) \times H^{-1}(\Omega)$ for which, for any $m\in\N$ and any control $f\in H^{-m}(0,T)$, the solution $y$ to system \eqref{eq:control}
doesn't satisfy \eqref{eq:control_final}.
\end{thm}
The proof of this result can be performed arguing by contraddiction. In particular, one can show that, under the controllability assumption, the map
$$
\begin{array}{rcl}
S\,:\; L^2(\Omega) \times H^{-1}(\Omega) & \longrightarrow & H^{-\infty}(0,T)\,:=\,\bigcup_{m\geq0}H^{-m}(0,T) \\[1ex]
(y_0\,,\,\omega\,y_1) & \mapsto & f
\end{array}
$$
is linear and continuous, in the sense that
$$
\left\|(y_0\,,\,\omega\,y_1)\right\|_{L^2 \times H^{-1}}\,\leq\,C\,\|f\|_{H^{-m}(0,T)}
$$
if $f\in H^{-m}(0,T)\!\setminus\!H^{-m+1}(0,T)$. In particular, this last inequality is equivalent to observability estimates. We refer to \cite{Castro-Z} (see in particular the proof of Theorem 3) for more details.



\begin{rem} \label{r:stab}
Observability inequalities we proved in the previous section have a direct counterpart also at level of stabilization of the wave equation in presence of a damping term (see also paper \cite{V-Z} for the case of networks).

In particular, using the same techniques as in \cite{V-Z} (see also \cite{B-S}) classical observability estimates of the form \eqref{est:obs_Z} would lead to an exponential decay of the energy, while observability estimates with finite loss as in \eqref{est:obs_log} would give a polinomial decay.

However, the exact statements and the corresponding proofs go beyond the scope of this paper: hence, we prefer not to address these questions.
\end{rem}

\section{Further comments and results} \label{s:further}
In this section we present some further results, closely related to the ones presented above, and discuss some open problems.

\begin{itemize}
\item[\textbf{(i)}] \textbf{The multi-dimensional case}

Let us spend a few words on the multi-dimensional case:
\begin{equation} \label{eq:control_N}
\left\{\begin{array}{ll}
        \omega(x)\d_t^2y\,-\,\Delta_xy\,=\,0 & \quad\mbox{ in }\;\Omega\times\,]0,T[ \\[1ex]
        y(t,\,\cdot\,)_{|\d\Omega}\,= f(t) & \quad\mbox{ in }\;]0,T[ \\[1ex]
        y(0,x)\,=\,y_0(x)\,,\quad \d_ty(0,x)\,=\,y_1(x) & \quad\mbox{ in }\;\Omega\,,
       \end{array}
\right.
\end{equation}
on a (smooth) bounded domain $\Omega\subset\R^N$, $N\geq2$.

In this instance, internal and boundary observability estimates read respectively as follows:
\begin{eqnarray*}
E(0) & \leq & C\,\int^T_0\int_\Theta\biggl(\omega(x)\,\bigl|\d_tu(t,x)\bigr|^2\,+\,\bigl|\nabla u(t,x)\bigr|^2\biggr)\,dx\,dt \\
E(0) & \leq & C\,\int^T_0\int_{\d\Omega}\bigl|\d_\nu u(t,x)\bigr|^2\,d\sigma\,dt\,,
\end{eqnarray*}
where $\Theta$ is an open subset of $\Omega$, $\d_\nu$ is the normal derivative and $d\sigma$ is the $(N-1)$-dimensional measure at the boundary $\d\Omega$.

At this level, the questions which arise in dimension bigger than $1$ are essentially two. The first one is finding the ``form'' of the characteristics for smooth coefficients, so that they are well defined and the dynamics follows the rays of geometric optics. This is related with the problem of verifiyng the GCC for subdomains of $\Omega$.

The second concerns instead the minimal regularity assumptions on the coefficients in order to define characteristics.
Besides the problems of concentration of rays, even assuming the GCC to be fulfilled by the observability domain, the microlocal analysis tools used in the study of the control problem for \eqref{eq:control_N} require much more smoothness on the coefficients. As a matter of fact, one has to ask for $\mc{C}^2$ smoothness in order to define and solve the corresponding Hamiltonian system, propagate informations along rays and get suitable observability estimates.

On the other hand, under some stronger geometric conditions than the GCC one on the observability domain, Carleman estimates allow to go down up to $\mc{C}^1$ regularity. This was done e.g. in paper \cite{D-Z-Z} by Duyckaerts, Zhang and Zuazua, thanks to a refined Carleman estimate for hyperbolic operators with potentials.

Therefore, at present, it is hard to expect positive results in the same spirit of the ones stated so far in section \ref{s:intro-results}. However, the issue of \cite{D-Z-Z} suggest us that Carleman estimates could be a good approach to tackle this problem. At this level, maybe it would be necessary to combine them with Littlewood-Paley decomposition (see also papers \cite{D-E} and \cite{D-L}, where dyadic decomposition is applied in observability problems), in order to take into account the loss of regularity, due to the bad behaviour of the coefficients, at different frequencies. This issue will be matter of next studies.

Nevertheless, at present we can state some negative results: repeating the construction of \cite{Castro-Z} (see in particular Section 6), it is
 possible
to see that, by separation of variables, the counterexamples we provided in the previous section work also in the instance of higher dimensions.
However, as remarked in \cite{Castro-Z} for H\"older continuity, these results do not have a direct counterpart in the setting of geometric optics and microlocal analysis techniques (we are even below the $\mc{C}^1$ regularity hypothesis).

\item[\textbf{(ii)}] \textbf{About the transport equation}

Due to the D'Alambert's formula for solutions to the $1$-dimensional wave equation
$$
\d^2_tu(t,x)\,-\,c^2\,\d^2_xu(t,x)\,=\,0
$$
($c\in\R$ constant), it is
 natural to look at the transport equation (still in dimension $1$)
$$
\d_tf\,+\,v\,\d_xf\,=\,0 \leqno{(T)}
$$
as a ``toy-model'' for our problem. Here we assume that the transport velocity $v$ depends just on $x$, and that it has some special modulus of continuity.

Solutions of $(T)$ are constant along characteristics: if $f_{|t=0}=f_0$, then, at any time $t$ and any point $x$, $f(t,X(t,x))\equiv f_0(x)$, where $X$ is the flow map associated to $v$:
$$
X(t,x)\,=\,x\,+\,\int^t_0v\bigl(X(\tau,x)\bigr)\,d\tau\,.
$$
Hence, to understand the dynamics one can just focus on $X$, which solves the ordinary differential equation
\begin{equation} \label{eq:char}
X'(t,x)\,=\,v\bigl(X(t,x)\bigr)
\end{equation}
with the initial condition $X(0,x)=x$.

It is
 well-known that, if $v$ has an Osgood modulus of continuity, i.e. for any $0<y<1$
$$
\bigl|v(x+y)\,-\,v(x)\bigr|\,\leq\,C\,\mu(y)\,,\qquad\mbox{ with}\quad\int^1_0\frac{1}{\mu(s)}\,ds\,=\,+\infty\,,
$$
then equation \eqref{eq:char} is uniquely solvable, locally in time. However, the regularity of the solution strictly depends on the modulus of continuity $\mu$: for instance, in the Lipschitz case ($\mu(s)=s$) the initial smoothness is preserved, while in the log-Lipschitz instance (i.e. $\mu(s)=s\,\log s$) the flow $X$ loses regularity in the evolution.

Correspondingly, one infers the respective counterparts for the transport equation $(T)$.
We refer to chapter 3 of \cite{B-C-D} for a complete discussion and the exact statements, as well as for the original references.

Nevertheless, the general picture is still far to be completely well-understood. As a matter of fact, the problem of uniqueness recently called a lot of attentions (see in particular the works by Colombini and Lerner \cite{C-L_2002} and by Ambrosio \cite{Ambr} for $BV$ velocity fields).

\medbreak
On the control side, the intuition says that, as the evolution is driven by characteristic curves, it is
 sufficient to control just one extreme of these characteristics in order to control the whole system. However, this is true just for $\mc{C}^1$ (say Lipschitz) regularity: see for instance work \cite{G-L} by Guerrero and Lebeau, where the authors investigated also the cost of control in the vanishing viscosity framework.


By analogy with the study performed in the present paper, we plan to consider the transport problem for non-Lipschitz velocity fileds. We expect that a vanishing viscosity argument (as used by Guerrero and Lebeau) would lead also in this instance to observability estimates, eventually admitting some loss of derivatives.

Moreover, it would be interesting to see if some strange concentration phenomena, similar to the ones showed in subsection \ref{ss:first}, occur also for transport equations under low regularity assumption on the velocity filed $v$.

We plan to address all these questions in next studies.

\item[\textbf{(iii)}] \textbf{On Strichartz and eigenfunctions estimates}

Let us observe that, as done in \cite{Castro-Z} and \cite{Castro-Z_add} for H\"older continuous coefficients, the counterexamples we established in subsection \ref{ss:first} can be adapted to prove also the lack of some dispersive type estimates.

\medbreak
First of all, we consider \emph{Strichartz estimates} for the wave equation in the whole $\R^N$:
\begin{equation} \label{eq:we_N}
\left\{\begin{array}{ll}
        \omega(x)\d_t^2y\,-\,\Delta_xy\,=\,0 & \quad\mbox{ in }\;\R^N\times\,]0,+\infty[ \\[1ex]
        y(0,x)\,=\,y_0(x)\,,\quad \d_ty(0,x)\,=\,y_1(x) & \quad\mbox{ in }\;\R^N\,.
       \end{array}
\right.
\end{equation}
We say that the pair $(p,q)\in[2,+\infty]^2$ is admissible if
$$
\frac{1}{p}\,+\,\frac{N-1}{q}\,\leq\,\frac{N-1}{2}\,.
$$
If $\omega$ is constant and $N\geq2$, it is
 well-known that, for any admissible pair $(p,q)$ such that
$(N,p,q)\neq(3,2,+\infty)$, one has
\begin{equation} \label{est:strichartz}
\left\|y\right\|_{L^p([0,T];L^q(\R^N_x))}\,\leq\,C_T\,\left(\|y_0\|_{H^s(\R^N)}\,+\,\|y_1\|_{H^{s-1}(\R^N)}\right)\,,
\end{equation}
where $s$ is defined by
$$
s\,:=\,N\left(\frac{1}{2}\,-\,\frac{1}{q}\right)\,-\,\frac{1}{p}\,.
$$
Analogous estimates were proved by Tataru in \cite{Tat} in the case of variable coefficients, when $\omega\in\mc{C}^\alpha$, for some $0\leq\alpha\leq2$ (where we mean $L^\infty$ if $\alpha=0$, $\mc{C}^{0,\alpha}$ if $\alpha\in\,]0,1[\,$, the Lipschitz space $\mc{C}^{0,1}$ if $\alpha=1$ and $\mc{C}^{1,\alpha-1}$ if $1<\alpha\leq2$).
In particular, there it is
 proved that, under these hypothesis, \eqref{est:strichartz} holds true for a constant $C_T$ depending just on the $\mc{C}^\alpha$ norm of $\omega$, provided that
$$
s\,=\,N\left(\frac{1}{2}\,-\,\frac{1}{q}\right)\,-\,\frac{\sigma-1}{p}\,,\qquad\mbox{ with }\quad\sigma\,=\,\frac{2-\alpha}{2+\alpha}\,.
$$
In \cite{Smith-Tat}, by construction of couterexamples, Smith and Tataru proved that this result is sharp: if \eqref{est:strichartz} is fulfilled, then $s$ has to be greater than or equal to the previous value. By counterexamples of different nature (close to the first one presented in subsection \ref{ss:first}), in \cite{Castro-Z_add} Castro and Zuazua proved instead the following statement.
\begin{prop} \label{p:strichartz}
If $\omega\in L^\infty$ and inequality \eqref{est:strichartz} holds true for some constant $C>0$, then
\begin{equation} \label{eq:stric_L-inf}
s\,\geq\,N\left(\frac{1}{2}\,-\,\frac{1}{q}\right)\qquad\left(\;\mbox{or, in an equivalent way, }\;q\,\leq\,\frac{2\,N}{N\,-\,2\,s}\right)\,.
\end{equation}

If instead $\omega\in\mc{C}^{0,\alpha}$, for some $0<\alpha<1$, then
$$
s\,\geq\,N\left(\frac{1}{2}\,-\,\frac{1}{q}\right)\,-\,\frac{\alpha\,N}{2}\,.
$$
\end{prop}
Note that the first issue, for $\alpha=0$, is sharp, because the value of $s$ coincides with the one found in \cite{Smith-Tat}.

Repeating the same construction of the first couterexample in subsection \ref{ss:first} and arguing as in \cite{Castro-Z_add}, it is
 easy to see that an analogous result is true also under the hypothesis of Theorem \ref{th:LL+}. More precisely, we recall that the function $\mu(r)=r\psi(r)$ is a weight, in the sense of definition 3.1 of \cite{C-L} (see also section 5 of the same paper), and then
$$
\mu(r)\,\sim\,r^{k_0}\qquad\mbox{ for large }r\,,
$$
for some $k_0\in\N$. Therefore, a $\omega$ which fulfills \eqref{eq:omega-psi} is $\alpha$-H\"older continuous of any exponent $\alpha\in\,]0,1[\,$. So, repeating the same steps of the proof given in \cite{Castro-Z}, we can find the following statement.
\begin{thm} \label{th:strichartz}
Let $\omega$ be as in the hypothesis of Theorem \ref{th:LL+}, and suppose that inequality \eqref{est:strichartz} holds true for some constant $C>0$.
Then necessarily one has
\begin{equation} \label{est:stric_LL+}
s\,>\,N\left(\frac{1}{2}\,-\,\frac{1}{q}\right)\,-\,\frac{N}{2}\,.
\end{equation}
\end{thm}
Note that, due to the complicate modulus of continuity which comes into play, we are not able to refine this condition, as done e.g. in Proposition \ref{p:strichartz} for H\"older coefficients. Roughly speaking, there the choice $\veps_j\sim h_j^{-\alpha}$ was fundamental, while in our case
the presence of $\log h_j\;\psi(\log h_j)$ makes a direct comparison of growth impossible.
Note also that the strictly inequality is necessary.

In the instance of Theorem \ref{th:L-LL} an analogous result holds true.

\medbreak
We want to discuss here also another issue, strictly connected with Strichartz estimates. This time we restrict to the $N$-dimensional torus $\mbb{T}^N$, and we consider the eigenvalue problem
\begin{equation} \label{eq:sogge}
-\,\Delta \phi\,+\,\lambda^2\,\omega(x)\,\phi\,=\,0\qquad\mbox{ in }\quad\mbb{T}^N\,,
\end{equation}
with $0<\omega_*\leq\omega\leq\omega^*$ as usual. We are interested in \emph{Sogge estimates} for the projection operators over spectral clusters of eigenfunctions.

Let $\left(\lambda_n\right)_n$ be the sequence of eigenvalues, and $\left(\phi_n\right)_n$ the corresponding orthonormal basis of $L^2(\mbb{T}^N)$. For any $\lambda\in\R$, we consider the orthogonal projection operator $\Pi_\lambda$ onto the subspace generated by the eigenfunctions with frequencies in the range $[\lambda,\lambda+1[\,$:
$$
\Pi_\lambda f\,=\,\sum_{\lambda_n\in[\lambda,\lambda+1[}\bigl(f,\phi_n\bigr)\,\phi_n\,,
$$
where we denoted by $\bigl(\,\cdot\,,\,\cdot\,\bigr)$ the scalar product in $L^2(\mbb{T}^N)$.

For any $q\in[2,+\infty]$ and any $N\geq1$, we are interested in estimates of the type
\begin{equation} \label{est:sogge}
\left\|\Pi_\lambda f\right\|_{L^q(\mbb{T}^N)}\,\leq\,C\,\lambda^\g\,\left\|f\right\|_{L^2(\mbb{T}^N)}\,,
\end{equation}
where the exponent $\g$ may depend on the dimension $N$ and on the summability index $q$, but not on the coefficient $\omega$. We denote by $\g_m(N,q)$ the minimum value of $\g$ for which the previous inequality holds true.

Such a kind of estimates were proved for the first time by Sogge in \cite{Sogge}, for any second order elliptic operator with smooth coefficients, defined on smooth compact manifolds without boundary. More precisely, Sogge proved that \eqref{est:sogge} holds for $\g=\wtilde{\g}(N,q)$ and $q\in[q_N,+\infty]$, where we defined
$$
\wtilde{\g}(N,q)\,:=\,N\left(\frac{1}{2}\,-\,\frac{1}{q}\right)\,-\,\frac{1}{2}\qquad\mbox{ and }\qquad q_N\,:=\,\frac{2\,(N+1)}{N-1}\,;
$$
for $2\leq q\leq q_N$, \eqref{est:sogge} is still true, up to change the value of $\g$. This result was then extended by Smith in \cite{Smith_2006} to the case of $\mc{C}^{1,1}$ regularity.

Under lower smoothness assumptions, however, Sogge's estimate \eqref{est:sogge} can fail in the previous range of parameters $\g$ and $q$. For instance, in \cite{Smith-Sogge} and \cite{Smith-Tat} the authors constructed counterexamples for coefficients of the operator in $\mc{C}^\alpha(\mbb{T}^N)$, when $0\leq\alpha<2$ (coefficients Lipschitz continuous for $\alpha=1$).

On the other hand, in \cite{Smith_2006_2} Smith was able to recover some weakened version of inequality \eqref{est:sogge} when the coefficients are in the H\"older classes $\mc{C}^\alpha$, $1\leq\alpha<2$ (again, $\mc{C}^{0,1}$ for $\alpha=1$). In fact, he proved that, under these regularity hypothesis, Sogge's estimate still holds true when the parameters satisfy the following conditions:
$$
\g_m(N,q)\,\leq\,\wtilde{\g}(N,q)\,+\,\frac{2-\alpha}{q\,(2+\alpha)}\,\leq\,1\qquad\mbox{ and }\qquad q_N\,\leq\,q\,\leq\,+\infty\,.
$$
As far as we know, it is
 still open if \eqref{est:sogge} can be recovered under the previous constraints (which are not in contrast with the negative results of \cite{Smith-Tat}), for $0\leq\alpha<1$.
A partial result was however given in \cite{Castro-Z_add}.
\begin{prop} \label{p:sogge}
If $\omega\in\mc{C}^{0,\alpha}$, for $0\leq\alpha<1$ ($\omega\in L^\infty$ if $\alpha=0$) and $q_N< q\leq+\infty$, then necessarily we must have
$$
\g_m(N,q)\,\geq\,\wtilde{\g}(N,q)\,+\,\frac{1\,-\,\alpha\,N}{2}\,.
$$
\end{prop}
Moreover, the previous claim can be easily adapted to cover also the case of any compact manifold (with or without boundary).

For the same reasons explained in the case of Strichartz estimates, our first counterexample in subsection \ref{ss:first} can used to prove an analogous but rough statement, due to the general complicated form of the modulus of continuity of the coefficient.
\begin{thm} \label{t:sogge}
Let $\omega$ be as in the hypothesis of Theorem \ref{th:LL+}, $q_N< q\leq+\infty$ and suppose that \eqref{est:sogge} holds true. Then necessarily one has
$$
\g_m(N,q)\,>\,\wtilde{\g}(N,q)\,+\,\frac{1-N}{2}\,=\,N\,\left(\frac{1}{2}\,-\,\frac{1}{q}\right)\,-\,\frac{N}{2}\,.
$$
\end{thm}
The previous statement is a consequence of Proposition \ref{p:sogge}, keeping in mind that $\omega$ is $\alpha$-H\"older continuous for any $\alpha\in\,]0,1[\,$.

An analogous result holds true also under the hypothesis of Theorem \ref{th:L-LL}.
\end{itemize}

\end{document}